\documentclass[a4paper,reqno]{amsart}
\usepackage{indentfirst}
\usepackage{amsfonts}
\usepackage{amscd}
\usepackage{amsthm}
\usepackage{amssymb}
\usepackage{amsmath}
\usepackage{thmtools}
\usepackage{epsfig}
\usepackage{graphicx}
\usepackage{slashed}
\usepackage[all,cmtip]{xy}
\usepackage{tikz}
\usepackage{todonotes}
\usetikzlibrary{decorations.markings}
\usepackage{enumitem}
\usepackage{setspace}
\usepackage{hyperref}

\usetikzlibrary{matrix,arrows,positioning}

\newcommand{\leftexp}[2]{{\vphantom{#2}}^{#1}{#2}}
\DeclareMathAlphabet{\mathpzc}{OT1}{pzc}{m}{it}

\def\itemNum$#1${\item $\displaystyle#1$
   \hfill\refstepcounter{equation}(\theequation)}

\newcommand*{\DashedArrow}[1][]{\mathbin{\tikz [baseline=-0.25ex,-latex, dashed,#1] \draw [#1] (0pt,0.5ex) -- (1.3em,0.5ex);}}%
\newcommand*{\dashtwoheadrarrow}{\DashedArrow[->>,densely dashed]}

\newtheorem{Lem}{Lemma}[section]
\newtheorem{Prop}[Lem]{Proposition}

\newtheorem*{Def}{Definition}

\theoremstyle{plain}
\newtheorem{Thm}[Lem]{Theorem}

{\unskip\nobreak\hskip 1em plus 1fil\nobreak$\Box$
\parfillskip=0pt%
\endtrivlist}

\theoremstyle{definition}
\declaretheorem[numbered=no,name=Example,qed={\lower-0.3ex\hbox{$\triangleleft$}}]{Ex}
\newtheorem*{Rem}{Remark}

\newcommand{\Hom}{\text{\textnormal{Hom}}}
\newcommand{\Ext}{\text{\textnormal{Ext}}}
\DeclareMathOperator{\im}{im}
\newcommand{\sbt}{\,\begin{picture}(-1,1)(-1,-3)\circle*{2.5}\end{picture}\ }

\begin{document}

\title{The Hall module of an exact category with duality}

\author[M.\,B. Young]{Matthew B. Young}
\address{Department of Mathematics\\
The University of Hong Kong\\
Pokfulam, Hong Kong}
\email{myoung@maths.hku.hk}

\date{\today}

\keywords{Representations of quivers, Hall algebras, quantum groups.}
\subjclass[2010]{Primary: 16G20 ; Secondary 17B37}

\begin{abstract}
We construct from a finitary exact category with duality $\mathcal{A}$ a module over its Hall algebra, called the Hall module, encoding the first order self-dual extension structure of $\mathcal{A}$. We study in detail Hall modules arising from the representation theory of a quiver with involution. In this case we show that the Hall module is naturally a module over the specialized reduced $\sigma$-analogue of the quantum Kac-Moody algebra attached to the quiver. For finite type quivers, we explicitly determine the decomposition of the Hall module into irreducible highest weight modules.
\end{abstract}

\maketitle

\setcounter{footnote}{0}

\section*{Introduction}
\label{sec:intro}

Let $\mathcal{A}$ be an abelian category with finite $\Hom$ and $\Ext^1$ sets, called finitary below. In \cite{ringel1990} Ringel defined the Hall algebra $\mathcal{H}_{\mathcal{A}}$, an associative algebra whose multiplication encodes the first order extension structure of $\mathcal{A}$. There is also a coalgebra structure on $\mathcal{H}_{\mathcal{A}}$ which, if $\mathcal{A}$ is hereditary, makes $\mathcal{H}_{\mathcal{A}}$ into a (twisted) bialgebra \cite{green1995}. The category $Rep_{\mathbb{F}_q}(Q)$ of representations of a quiver over a finite field is a finitary hereditary category. The corresponding Hall algebra $\mathcal{H}_Q$ contains a subalgebra isomorphic to the positive part of the quantum Kac-Moody algebra associated to $Q$, specialized at $\sqrt{q}$ \cite{ringel1990v3}, \cite{green1995}. A second example of a finitary hereditary category is the category of coherent sheaves over a smooth projective curve $X$ defined over $\mathbb{F}_q$. In the simplest case, $X=\mathbb{P}^1$, the Hall algebra contains a subalgebra isomorphic to a positive part of the quantum affine algebra $U_{\sqrt{q}}(\hat{\mathfrak{sl}}_2)$ \cite{kapranov1997}. More generally, Hall algebras can be defined for exact categories \cite{hubery2006} and often behave similarly to quantum nilpotent groups \cite{berenstein2012}.

In this paper we introduce an analogue of the Hall algebra when objects of $\mathcal{A}$ are allowed to carry non-degenerate quadratic forms. To do this, we work with exact categories with duality. In this setting, a self-dual object is an object of $\mathcal{A}$ together with a symmetric isomorphism with its dual. Instead of using extensions to define an algebra we use the self-dual extension structure of $\mathcal{A}$ to define a $\mathcal{H}_{\mathcal{A}}$-module, called the Hall module and denoted by $\mathcal{M}_{\mathcal{A}}$; see Theorem \ref{thm:hallModuleExact}. More precisely, a self-dual exact sequence is a diagram
\[
0 \rightarrow U \rightarrow M \dashrightarrow N \rightarrow 0
\]
presenting $U \in \mathcal{A}$ as an isotropic subobject of the self-dual object $M$ and presenting the self-dual object $N$ as the isotropic reduction of $M$ by $U$.  We also show that $\mathcal{M}_{\mathcal{A}}$ is naturally a $\mathcal{H}_{\mathcal{A}}$-comodule. It is important to be able to twist the Hall module so as to obtain modules over the Ringel-twisted Hall algebra. For example, the connection between quantum groups and Hall modules described below is most clear when using this twist. In Theorem \ref{thm:descendtoGroth} we construct such a module twist using an integer valued function $\mathcal{E}$ on the Grothendieck group of $\mathcal{A}$. The function $\mathcal{E}$ plays the role of the Euler form for categories with duality and is therefore of independent interest. In Theorem \ref{thm:sdRiedtmann}, we prove an identity relating $\mathcal{E}$, the Euler form and the stacky number of self-dual extensions in the case that $\mathcal{A}$ is hereditary. The proof develops some basic self-dual homological algebra and uses the combinatorics of self-dual analogues of Grothendieck's \textit{extensions panach\'{e}es} \cite{grothendieck1972}, \cite{bertrand2013}.

In Section \ref{sec:hallModQuantGrp} we study Hall modules arising from the representation theory of a quiver with contravariant involution $(Q,\sigma)$. From the involution and a choice of signs we define a duality structure on $Rep_{\mathbb{F}_q}(Q)$, with $q$ odd. For particular signs, the self-dual objects coincide with the orthogonal and symplectic representations of Derksen and Weyman \cite{derksen2002}. The module and comodule structures are incompatible in that $\mathcal{M}_Q$ is not a Hopf module, even in a twisted sense. In Theorem \ref{thm:redModuleStructure}, we instead show that the action and coaction of the simple representations $[S_i] \in \mathcal{H}_Q$ make $\mathcal{M}_Q$, with its $\mathcal{E}$-twisted module structure, a module over $B_{\sigma}(\mathfrak{g}_Q)$, the specialized reduced $\sigma$-analogue of $U_v(\mathfrak{g}_Q)$.  The proof is combinatorial in nature and involves counting configurations of pairs of self-dual exact sequences, in the spirit of Green's proof of the bialgebra structure of the Hall algebra \cite{green1995}. We describe in Theorem \ref{thm:hallModDecomp} the decomposition of $\mathcal{M}_Q$ into irreducible highest weight $B_{\sigma}(\mathfrak{g}_Q)$-modules. The generators are cuspidal elements of $\mathcal{M}_Q$, i.e. elements that are annihilated by the coaction of each $[S_i]$. The proof relies on a canonically defined non-degenerate bilinear form on the Hall module and a characterization of irreducible highest weight modules due to Enomoto and Kashiwara \cite{enomoto2008}.

In Section \ref{sec:ftHallMod} we restrict attention to finite type quivers. Unlike ordinary quiver representations, self-dual representations in general have non-trivial $\overline{\mathbb{F}}_q \slash \mathbb{F}_q$-forms. We extend results of \cite{derksen2002} (over algebraically closed fields) to explicitly describe all such forms and classify the indecomposable self-dual $\mathbb{F}_q$-representations.  The classification is summarized in Theorem \ref{thm:sdGabriel} where a partial root theoretic interpretation of the indecomposables is given. The main application of this result is to the explicit decomposition of Hall modules of finite type quivers into irreducible highest weight $B_{\sigma}(\mathfrak{g}_Q)$-modules; see Theorems \ref{thm:noKFormIrred} and \ref{thm:noHypCusp}. The generators are written as alternating sums of the $\overline{\mathbb{F}}_q \slash \mathbb{F}_q$-forms of self-dual indecomposables.

In \cite{enomoto2009}, Enomoto proved a result related to Theorems \ref{thm:redModuleStructure} and \ref{thm:hallModDecomp}, showing that induction and restriction operators along $[S_i]$ endow the Grothendieck group of a category of perverse sheaves on the moduli stack of complex orthogonal representations with the structure of a highest weight $B_{\sigma}(\mathfrak{g}_Q)$-module. In the terminology of the present paper, the weight module in \cite{enomoto2009} is generated by the trivial orthogonal representation, whereas Theorems \ref{thm:redModuleStructure} and \ref{thm:hallModDecomp} hold for arbitrary dualities and describe the decomposition of the entire Hall module. The techniques used by Enomoto generalize Lusztig's geometric approach to canonical bases \cite{lusztig1990} and are completely different from those used in this paper. The existence of both approaches suggests a self-dual analogue of Lusztig's purity result \cite{lusztig1998} for multiplicity complexes of perverse sheaves. This would provide a direct link between \cite{enomoto2009} and the present paper.

\textbf{Notations and assumptions:} In this paper, all fields are assumed to have characteristic different from two. In particular, the number of elements $q$ in the finite field $\mathbb{F}_q$ is necessarily odd. All categories are assumed to be essentially small. We write $Iso(\mathcal{A})$ for the set of isomorphism classes of objects of a category $\mathcal{A}$.

\subsubsection*{Acknowledgements}
The author would like to thank Cheng Hao for helpful comments during the preparation of this work and Michael Movshev for his insights and encouragement. The author was partially supported by an NSERC Postgraduate Scholarship.

\section{The Hall algebra of an exact category}
\label{sec:hallAlgExact}

Let $\mathcal{A}$ be an exact category in the sense of Quillen \cite{quillen1973}. In particular, $\mathcal{A}$ is additive and is equipped with a collection $\mathcal{F}$ of kernel-cokernel pairs $(i, \pi)$, called short exact sequences and denoted
\begin{equation}
\label{eq:shortES}
U \overset{i}{\rightarrowtail} X \overset{\pi}{\twoheadrightarrow} V,
\end{equation}
satisfying a collection of axioms \cite{quillen1973}, \cite{buhler2010}. For example, abelian categories and their extension-closed full subcategories have canonical exact structures.  A morphism $i$ is called an admissible monic if it occurs in a pair $(i, \pi) \in \mathcal{F}$.

Denote by $\underline{\mathcal{F}}^X_{U,V}$ the set of short exact sequences of the form \eqref{eq:shortES}. Assume that $\mathcal{A}$ is finitary, that is, for all $U,V \in \mathcal{A}$, the set $\Hom(U,V)$ is finite and $\underline{\mathcal{F}}^X_{U,V}$ is non-empty for only finitely many $X \in Iso(\mathcal{A})$. The Hall numbers are then the cardinalities
\[
F^X_{U,V} = \vert \{ \tilde{U} \subset X \; \vert \; \tilde{U} \simeq U, \; X \slash \tilde{U} \simeq V  \} \vert,
\]
where the subobjects $\tilde{U}$ are required to be admissible. Setting $a(U) = \vert Aut(U) \vert$ we have $\vert \underline{\mathcal{F}}^X_{U,V} \vert = a(U) a(V) F^X_{U,V}$.

Fix an integral domain $R$ containing $\mathbb{Q}$, a unit $\nu \in R$ and a bilinear form $c$ on $K(\mathcal{A})$, the Grothendieck group of $\mathcal{A}$. The Hall algebra of $\mathcal{A}$ is the free $R$-module with basis $Iso(\mathcal{A})$,
\[
\mathcal{H}_{\mathcal{A}} = \bigoplus_{U \in Iso(\mathcal{A}) } R [U].
\]
The associative multiplication is given by \cite{ringel1990}, \cite{hubery2006}
\begin{equation}
\label{eq:hallMult}
[U] [V] = \nu^{c(V,U)}\sum_{X \in Iso(\mathcal{A})} F^X_{U,V} [X].
\end{equation}
Similarly, $\mathcal{H}_{\mathcal{A}}$ is a topological coassociative coalgebra with coproduct \cite{green1995}
\[
\Delta [X] = \sum_{U,V \in Iso(\mathcal{A})} \nu^{c(V,U)}\frac{ a(U) a(V)}{a(X)} F^X_{U,V} [U] \otimes [V].
\]
In general, the coproduct takes values in the completion $\mathcal{H}_{\mathcal{A}} \hat{\otimes}_R \mathcal{H}_{\mathcal{A}}$ consisting of all formal linear combinations $\sum_{U,V} c_{U,V} [U] \otimes [V]$; see \cite{schiffmann2006}. Both the product and coproduct respect the natural $K(\mathcal{A})$-grading of $\mathcal{H}_{\mathcal{A}}$.

Suppose that $\mathcal{A}$ is $\mathbb{F}_q$-linear. If $\mathcal{A}$ has finite homological dimension and finite dimensional $\Ext^i$ groups, $i\geq 0$, then its Euler form is the bilinear form on $K(\mathcal{A})$ defined by
\[
\langle U,V \rangle = \sum_{i \geq 0} (-1)^i \hbox{dim}_{\mathbb{F}_q} \, \Ext^i(U,V).
\]
Its symmetrization is denoted $( \cdot, \cdot)$. With the choices $\nu = \sqrt{q}^{-1}$, $R = \mathbb{Q}[\nu, \nu^{-1}] \subset \mathbb{R}$ and $c= - \langle \cdot, \cdot \rangle$, $\mathcal{H}_{\mathcal{A}}$ is called the Ringel-Hall algebra of $\mathcal{A}$.

The following fundamental result asserts the compatibility of the product and coproduct when $\mathcal{A}$ is hereditary, i.e. of homological dimension at most one.

\begin{Thm}[\cite{green1995}]
\label{thm:green}
Let $\mathcal{H}_{\mathcal{A}}$ be the Ringel-Hall algebra of a hereditary abelian category $\mathcal{A}$. Equip $\mathcal{H}_{\mathcal{A}} \hat{\otimes}_R \mathcal{H}_{\mathcal{A}}$ with the algebra structure given on homogeneous elements by
\[
(x \otimes y)  (z \otimes w) = \nu^{-(y,z)} xz \otimes yw.
\]
Then $\Delta: \mathcal{H}_{\mathcal{A}} \rightarrow \mathcal{H}_{\mathcal{A}} \hat{\otimes}_R \mathcal{H}_{\mathcal{A}}$ is an algebra homomorphism.
\end{Thm}

Finally, in \cite{green1995}, Green defined an $R$-valued non-degenerate symmetric bilinear form on $\mathcal{H}_{\mathcal{A}}$ by $\displaystyle ( [U], [V] )_{\mathcal{H}} = \frac{\delta_{U,V}}{a(U)}$. This form satisfies
\[
( x \otimes y , \Delta z )_{\mathcal{H} \otimes \mathcal{H}}  =  ( x y , z )_{\mathcal{H}}, \;\;\;\ x,y,z  \in \mathcal{H}_{\mathcal{A}}
\]
where $(x \otimes y , x^{\prime} \otimes y^{\prime} )_{\mathcal{H} \otimes \mathcal{H}} = (x,x^{\prime})_{\mathcal{H}} (y,y^{\prime})_{\mathcal{H}}$.

The category $Rep_{\mathbb{F}_q}(Q)$ satisfies the conditions of Theorem \ref{thm:green}. Its Hall algebra, discussed in Section \ref{sec:hallModQuantGrp}, is closely related to the quantum Kac-Moody algebra attached to $Q$. We describe here a second example only briefly. The reader can find many examples of Hall algebras in \cite{schiffmann2006}.

\begin{Ex}[\cite{kapranov1997}]
Let $X$ be a smooth projective curve defined over $\mathbb{F}_q$. Theorem \ref{thm:green} applies to the category of coherent sheaves over $X$. The full exact subcategory of vector bundles defines a subalgebra $\mathcal{H}_{Vect_X} \subset \mathcal{H}_{Coh_X}$. Using the ad\`{e}lic description of the stack of vector bundles over $X$, $\mathcal{H}_{Vect_X}$ can be interpreted as the unramified automorphic forms for $GL$ defined over $\mathbb{F}_q(X)$ with multiplication given by the parabolic Eisenstein series map. Incorporating torsion sheaves gives a Hall algebraic realization of Hecke operators. The quadratic identities satisfied by cusp eigenforms become quadratic relations in $\mathcal{H}_{Coh_X}$. These relations imply that $\mathcal{H}_{Coh_{\mathbb{P}^1}}$ is isomorphic to the semidirect product of the Hall algebra of torsion sheaves (a tensor product of classical Hall algebras) with a negative part of the quantum affine algebra $U_{\sqrt{q}^{-1}}(\hat{\mathfrak{sl}}_2)$. In higher genus the quadratic relations no longer determine $\mathcal{H}_{Coh_X}$. See \cite{schiffmann2011}, \cite{kapranov2012} for results in higher genus.
\end{Ex}

\section{The Hall module of an exact category with duality}
\label{sec:hallModExact}

\subsection{Exact categories with duality}
We recall some basics of the theory of exact categories with duality. See \cite{balmer2005} for details.

\begin{Def}
\begin{enumerate}[leftmargin=0cm,itemindent=.6cm,labelwidth=\itemindent,labelsep=0cm,align=left]
\item
An exact category with duality is a triple $(\mathcal{A}, S, \Theta)$ consisting of an exact category $\mathcal{A}$, an exact contravariant functor $S: \mathcal{A} \rightarrow \mathcal{A}$ and an isomorphism $\Theta: 1_{\mathcal{A}} \xrightarrow[]{\sim} S^2$ satisfying $S(\Theta_U) \Theta_{S(U)} = 1_{S(U)}$ for all $U \in \mathcal{A}$.

\item A self-dual object of $(\mathcal{A}, S, \Theta)$ is an object $N \in \mathcal{A}$ together with an isomorphism $\psi_N: N \xrightarrow[]{\sim} S(N)$ satisfying $S(\psi_N) \Theta_N = \psi_N$.
\end{enumerate}
\end{Def}

The notation $(N, \psi_N) \in \mathcal{A}_S$, or just $N \in \mathcal{A}_S$ if $\psi_N$ is understood, will indicate that $(N, \psi_N)$ is a self-dual object. We also sometimes refer to $\mathcal{A}$, instead of $(\mathcal{A}, S, \Theta)$, as an exact category with duality. We say that $N, N^{\prime} \in \mathcal{A}_S$ are isometric, notation $N \simeq_S N^{\prime}$, if there exists an isomorphism $\phi: N \xrightarrow[]{\sim} N^{\prime}$ satisfying $S(\phi) \psi_{N^{\prime}} \phi = \psi_N$. The set of isometry classes of self-dual objects $Iso(\mathcal{A}_S)$ is an abelian monoid under orthogonal direct sum.

\begin{Ex}
Let $X$ be a smooth variety and write $Vect_X$ for the exact category of vector bundles over $X$. Given $s \in \{\pm 1\}$ and a line bundle $\mathcal{L} \rightarrow X$, define a duality on $Vect_X$ by $S(\mathcal{V}) = \mathcal{V}^{\vee} \otimes_{\mathcal{O}_X} \mathcal{L}$, where $\mathcal{V}^{\vee}= \Hom_{\mathcal{O}_X}(\mathcal{V}, \mathcal{O}_X)$. Let $\Theta_{\mathcal{V}}: \mathcal{V} \xrightarrow[]{\sim} \mathcal{V}^{\vee \vee}$ be the signed evaluation map, given at the level of sections by $\Theta(f)(x) = s \cdot f(x)$, $x \in X$. Then $(Vect_X, S, \Theta)$ is an exact category with duality, the self-dual objects being $\mathcal{L}$-valued orthogonal or symplectic vector bundles over $X$.
\end{Ex}

\begin{Def}
Let $N \in \mathcal{A}_S$. An admissible monic $U \overset{i}{\rightarrowtail} N$ is called isotropic if $S(i) \psi_N i=0$ and if the induced monic from $U$ to its orthogonal $U^{\perp}:= \text{\textnormal{ker}}\, S(i)\psi_N $ is also admissible.
\end{Def}

We write $U \overset{\perp}{\subset} N$ when $U$ is an isotropic subobject of $N$. The following categorical version of orthogonal and symplectic reduction will be used throughout the paper.

\begin{Prop}[{\cite[Proposition 5.2]{quebbemann1979}}]
\label{prop:sdReduction}
If $i:U \rightarrowtail N$ is isotropic, then the reduction $N /\!/ U:= U^{\perp} \slash U$ inherits form $N$ a canonical self-dual structure, i.e. there exists a self-dual structure $\tilde{\psi}$ on $N /\!/ U$, unique up to isometry, making the following diagram commute:
\begin{equation}
\label{diag:sdExact}
\begin{tikzpicture}[every node/.style={on grid},  baseline=(current bounding box.center),node distance=1.9]
  \node (A) {$U$};
  \node (B) [right=of A]{$E$}; 
  \node (C) [right=of B]{$N /\!/ U$};
  \node (D) [below=of A]{$U$};
  \node (E) [right=of D]{$N$}; 
  \node (F) [right=of E]{$S(E)$};
  \node (G) [below=of E]{$S(U)$};
  \node (H) [right=of G]{$S(U)$};
  
  \draw[>->] (A)-- node [above] {\footnotesize $j$ }(B);
  \draw[->>] (B)-- node [above] {\footnotesize $\pi$ }  (C);
  \draw[>->] (D)-- node [above] {\footnotesize $i$ }  (E);
  \draw[->>] (E)-- node [above] {\footnotesize $S(k) \psi_N$ }  (F);
  \draw[double equal sign distance] (A)-- (D);
  \draw[double equal sign distance] (G)-- (H);
  \draw[>->] (B)-- node [left] {\footnotesize $k$ }  (E);
  \draw[>->] (C)-- node [right] {\footnotesize $S(\pi) \tilde{\psi}$ }  (F);
  \draw[->>] (E)-- node [left] {\footnotesize $S(i) \psi_N$ }  (G);
  \draw[->>] (F)-- node [right] {\footnotesize $S(j)$ }  (H);
\end{tikzpicture}
\end{equation}
Here $k$ is a kernel of $S(i) \psi_N$ and $\pi$ is a cokernel for the induced monic $j$.
\end{Prop}

Motivated by Proposition \ref{prop:sdReduction}, we make the following definition.

\begin{Def}
Given $U \in \mathcal{A}$, $M,N \in \mathcal{A}_S$, let $\underline{\mathcal{G}}^N_{U,M}$ be the set of equivalence classes of exact commutative diagrams $(E;i,j,k,\pi)$ of the form \eqref{diag:sdExact}, with $(N /\!/ U, \tilde{\psi})$ replaced with $(M, \psi_M)$. Two such diagrams $E, E^{\prime}$, are equivalent if there exists an isomorphism $E \xrightarrow[]{\sim} E^{\prime}$ making all appropriate diagrams commute.
\end{Def}

Elements of $\underline{\mathcal{G}}^N_{U,M}$  are called self-dual exact sequences and are written
\[
U \overset{i}{\rightarrowtail} N \overset{\pi}{\dashtwoheadrarrow} M.
\]

\subsection{Hall modules}
Let $\mathcal{A}$ be a finitary exact category with duality. For $U \in \mathcal{A}$ and $M,N \in \mathcal{A}_S$, define the self-dual Hall number by
\begin{equation}
\label{eq:sdHallNumSet}
G^N_{U,M} = \vert  \{ \tilde{U} \overset{\perp}{\subset} N \; \vert \; \tilde{U} \simeq U, \;\; N /\!/ \tilde{U} \simeq_S M \} \vert.
\end{equation}
Let $\mathcal{G}^N_{U,M}= \vert \underline{\mathcal{G}}^N_{U,M} \vert$ and $a_S(M) = \vert Aut_S(M) \vert$, the number of isometries of $M$.

\begin{Lem}
\label{lem:diffHallNums}
The equality $\displaystyle G^N_{U,M}= \frac{\mathcal{G}^N_{U,M}}{a(U) a_S(M)}$ holds.
\end{Lem}
\begin{proof}
The group $Aut(U) \times Aut_S(M)$ acts freely on $\underline{\mathcal{G}}^N_{U,M}$ by
\[
(g,h) \cdot (E;i,j,k,\pi) = (E;ig^{-1},jg^{-1},k,h\pi), \;\;\; (g,h) \in Aut(U) \times Aut_S(M).
\]
The map $(E;i,j,k,\pi) \mapsto  \im (i)$ induces a bijection from $\underline{\mathcal{G}}^N_{U,M} \slash Aut(U) \times Aut_S(M)$ to the set appearing on the right-hand side of equation \eqref{eq:sdHallNumSet}.
\end{proof}

\begin{Lem}
\label{lem:finiteSupport}
For fixed $U \in \mathcal{A}$ and $M \in \mathcal{A}_S$, the set $\underline{\mathcal{G}}^N_{U,M}$ is non-empty for at most finitely many $N \in Iso(\mathcal{A}_S)$.
\end{Lem}
\begin{proof}
Since $\mathcal{A}$ is finitary, for fixed $U$ and $M$ there only finitely many isomorphism types of $E$, and hence $N$, that can appear in diagram \eqref{diag:sdExact}. By $\Hom$-finiteness, any such $N$ admits at most finitely many self-dual structures.
\end{proof}

Let $\mathcal{M}_{\mathcal{A}}$ be the free $R$-module with basis $Iso(\mathcal{A}_S)$,
\[
\displaystyle \mathcal{M}_{\mathcal{A}}= \bigoplus_{M \in Iso(\mathcal{A}_S)} R [M].
\]
The next result defines a $\mathcal{H}_{\mathcal{A}}$-(co)module structure on $\mathcal{M}_{\mathcal{A}}$, which we call the Hall module of $\mathcal{A}$. For now, take $c=0$ in equation \eqref{eq:hallMult}.

\begin{Thm}
\label{thm:hallModuleExact}
The formulae
\[
[U] \star [M] = \sum_{N \in Iso(\mathcal{A}_S)} G^N_{U,M} [N]
\]
and
\[
\rho[N] = \sum_{U \in Iso(\mathcal{A})} \sum _{M \in Iso(\mathcal{A}_S)} \frac{a(U) a_S(M)}{a_S(N)} G^N_{U,M} [U] \otimes [M]
\]
make $\mathcal{M}_{\mathcal{A}}$ a left $\mathcal{H}_{\mathcal{A}}$-module and topological left $\mathcal{H}_{\mathcal{A}}$-comodule.
\end{Thm}

\begin{proof}
Lemmas \ref{lem:diffHallNums} and \ref{lem:finiteSupport} imply that the above formulae are well-defined. A direct calculation shows that the $\mathcal{H}_{\mathcal{A}}$-action is associative if and only if
\begin{equation}
\label{eq:associativity}
\sum_{W \in Iso(\mathcal{A})} F^W_{U,V} G^N_{W,M} = \sum_{P \in Iso(\mathcal{A}_S)} G^N_{U, P} G^P_{V,M}, \;\;\;\;\;\; U, V \in \mathcal{A}, \; M,N \in \mathcal{A}_S.
\end{equation}
Interpreting this equation in terms of isotropic filtrations shows that it is equivalent to a self-dual version of the Second Isomorphism theorem. Precisely, this theorem says that, for fixed $U \overset{\perp}{\subset} N$, the map $V \mapsto V \slash U$ defines a bijection
\[
\{ V \overset{\perp}{\subset} N \;  \vert \; U \subset V \} \longleftrightarrow \{ \tilde{V} \overset{\perp}{\subset} N /\!/ U  \}
\]
satisfying
\[
(N /\!/ U) /\!/ (V \slash U) \simeq_S N /\!/ V.
\]
The self-dual Second Isomorphism theorem is proved in \cite[Proposition 6.5]{quebbemann1979} for abelian categories. The same argument applies to exact categories once admissibility of all subobjects involved is verified, which is straightforward.

Turning to the comodule structure, we first show that the composition
\[
(1 \otimes \rho) \circ \rho : \mathcal{M}_{\mathcal{A}} \rightarrow \mathcal{H}_{\mathcal{A}} \hat{\otimes} \mathcal{H}_{\mathcal{A}} \hat{\otimes} \mathcal{M}_{\mathcal{A}}
\]
is well-defined, the completion consisting again of all formal linear combinations. The map $(\Delta \otimes 1) \circ \rho$ is dealt with similarly. For any $\xi \in \mathcal{M}_{\mathcal{A}}$, the terms in $\rho(\xi)$ contributing to the coefficient of $[U_1] \otimes [U_2] \otimes [M]$ in $(1 \otimes \rho) \circ \rho (\xi)$ are proportional to $[U_1] \otimes [N]$ where $N /\!/ U_2 \simeq_S M$, and are therefore finite in number by Lemma \ref{lem:finiteSupport}. A direct calculation now shows that coassociativity is equivalent to equation \eqref{eq:associativity}, which has already been established.
\end{proof}

We now introduce a generalized grading on the Hall module. Recall that $N \in \mathcal{A}_S$ is called metabolic if it contains a Lagrangian, i.e. an isotropic subobject $U$ with $U^{\perp} = U$. For example, the hyperbolic object on $U \in \mathcal{A}$, defined as
\[
H(U)= \left( U \oplus S(U), \left( \begin{smallmatrix} 0 & 1_{S(U)} \\ \Theta_U & 0 \end{smallmatrix} \right) \right) \in \mathcal{A}_S,
\]
is metabolic, having both $U$ and $S(U)$ as Lagrangians.

\begin{Def}[see \cite{balmer2005}]
\begin{enumerate}[leftmargin=0cm,itemindent=.6cm,labelwidth=\itemindent,labelsep=0cm,align=left]
\item The Grothendieck-Witt group $GW(\mathcal{A})$ is the Grothendieck group of $Iso(\mathcal{A}_S)$ modulo the relation $\vert N \vert = \vert H(U) \vert$ whenever $N$ is metabolic with Lagrangian $U$.

\item The Witt group $W(\mathcal{A})$ is the abelian monoid $Iso(\mathcal{A}_S)$ modulo the submonoid of metabolic objects.
\end{enumerate}
\end{Def}

The groups $GW(\mathcal{A})$ and $W(\mathcal{A})$ give two $R$-module decompositions of $\mathcal{M}_{\mathcal{A}}$,
\[
\mathcal{M}_{\mathcal{A}} = \bigoplus_{\gamma \in GW(\mathcal{A})} \mathcal{M}_{\mathcal{A}}(\gamma), \;\;\;\;\;\; \mathcal{M}_{\mathcal{A}} = \bigoplus_{w \in W(\mathcal{A})} \mathcal{M}_{\mathcal{A}}(w),
\]
where $\mathcal{M}_{\mathcal{A}}(\gamma)$ is spanned by self-dual objects of class $\gamma \in GW(\mathcal{A})$ and similarly for $\mathcal{M}_{\mathcal{A}}(w)$. Denote by 
\[
H:K(\mathcal{A}) \rightarrow GW(\mathcal{A})
\]
the group homomorphism defined at the level of objects by $U \mapsto H(U)$.

\begin{Prop}
\label{prop:wittGradings}
The homomorphism $H$ intertwines the $K(\mathcal{A})$ and $GW(\mathcal{A})$-gradings of $\mathcal{H}_{\mathcal{A}}$ and $\mathcal{M}_{\mathcal{A}}$: for all $\alpha \in K(\mathcal{A})$ and $\gamma \in GW(\mathcal{A})$
\[
\mathcal{H}_{\mathcal{A}}(\alpha) \star \mathcal{M}_{\mathcal{A}}(\gamma) \subset \mathcal{M}_{\mathcal{A}}(H(\alpha) + \gamma).
\]
Moreover, for each $w \in W(\mathcal{A})$, $\mathcal{M}_{\mathcal{A}}(w) \subset \mathcal{M}_{\mathcal{A}}$ is an $\mathcal{H}_{\mathcal{A}}$-submodule.  Analogous statements hold for the comodule structure.
\end{Prop}
\begin{proof}
The first statement follows from that fact that if $U \overset{\perp}{\subset} N$, then in $GW(\mathcal{A})$ we have the identity
\[
\vert N \vert= \vert N /\!/ U \vert + \vert H(U) \vert.
\]
See \cite{quebbemann1979}. The second statement now follows from the exact sequence of abelian groups
\begin{equation}
\label{eq:GWexact}
K(\mathcal{A}) \xrightarrow[]{H} GW(\mathcal{A}) \rightarrow W(\mathcal{A}) \rightarrow 0.
\end{equation}
\end{proof}

To extend Theorem \ref{thm:hallModuleExact} to $c$-twisted Hall algebras, suppose that we are given a function $\tilde{c}: GW(\mathcal{A}) \times K(\mathcal{A}) \rightarrow \mathbb{Z}$ satisfying, for all $\alpha,\beta \in K(\mathcal{A})$ and $\gamma \in GW(\mathcal{A})$,
\begin{equation}
\label{eq:twistCocycle}
c(\alpha, \beta) + \tilde{c}(\gamma, \alpha + \beta) = \tilde{c}(\gamma, \alpha) + \tilde{c}(\gamma + H(\alpha), \beta).
\end{equation}
This guarantees that the twisted action
\[
[U] \star [M] = \nu^{\tilde{c}(M,U)}\sum_{N \in Iso(\mathcal{A}_S)} G^N_{U,M} [N]
\]
makes $\mathcal{M}_{\mathcal{A}}$ a module over the $c$-twisted Hall algebra. The comodule structure is similarly modified.

In this paper we are primarily interested in the Ringel-Hall algebra. We therefore seek a function $\tilde{c}$ compatible with $c= - \langle \cdot, \cdot \rangle$. For each $U \in \mathcal{A}$ and $i \geq 0$, the pair $(S, \Theta)$ generates a linear action of $\mathbb{Z}_2$ on $\Ext^i(S(U),U)$. We write $\Ext^i(S(U),U)^{ p S}$ for the subspace of symmetric ($p=1$) or skew-symmetric ($p=-1$) extensions with respect to this action.

\begin{Thm}
\label{thm:descendtoGroth}
Let $\mathcal{A}$ be a $k$-linear abelian category of finite homological dimension with finite dimensional  $\Ext^i$-groups. Then the function $\mathcal{E}: Iso(\mathcal{A}) \rightarrow \mathbb{Z}$ given by
\[
\mathcal{E}(U) = \sum_{i \geq 0} (-1)^i \dim_k \Ext^i(S(U),U)^{(-1)^{i+1} S}
\]
descends to $K(\mathcal{A})$. Moreover, the function
\[
\tilde{c}(M,U) = -\langle M,U \rangle - \mathcal{E}(U)
\]
satisfies equation \eqref{eq:twistCocycle} with $c = -\langle \cdot, \cdot \rangle$.
\end{Thm}
\begin{proof}
Applying the bifunctor $\Hom(-, -)$ to the short exact sequence
\[
0 \rightarrow U \rightarrow W \rightarrow V \rightarrow 0
\]
and its dual gives six long exact sequences fitting into the following diagram:
\begin{center}
\begin{tikzpicture}[descr/.style={fill=white,inner sep=1.3pt}]
        \matrix (m) [
            matrix of math nodes,
            row sep=1em,
            column sep=2.5em,
            text height=1.5ex, text depth=0.25ex
        ]
        {  & 0 & 0 & 0 & \hbox{} \\
        	 0 & \Hom(S(U),U) & \Hom(S(U),W) & \Hom(S(U),V) & \hbox{} \\
           0 & \Hom(S(W),U) & \Hom(S(W),W) & \Hom(S(W),V) & \\
           0 & \Hom(S(V),U) & \Hom(S(V),W) & \Hom(S(V),V) & \\
           & \Ext^1(S(U),U) & \Ext^1(S(U),W) & \Ext^1(S(U),V) & \\
           & \Ext^1(S(W),U) & \Ext^1(S(W),W) & \Ext^1(S(W),V) & \\
           & \vdots & \vdots & \vdots & \\
        };
        \path[->]
        (m-2-1) edge (m-2-2)
        (m-2-2) edge (m-2-3)
        (m-2-3) edge (m-2-4)
        (m-3-1) edge (m-3-2)
        (m-3-2) edge (m-3-3)
        (m-3-3) edge (m-3-4)
        (m-4-1) edge (m-4-2)
        (m-4-2) edge (m-4-3)
        (m-4-3) edge (m-4-4)        
        (m-5-2) edge (m-5-3)
        (m-5-3) edge (m-5-4)
        (m-6-2) edge (m-6-3)
        (m-6-3) edge (m-6-4)        
        
        (m-1-2) edge (m-2-2)
        (m-2-2) edge (m-3-2)
        (m-3-2) edge (m-4-2)
        (m-4-2) edge (m-5-2)
        (m-5-2) edge (m-6-2)
        (m-6-2) edge (m-7-2)
        (m-1-3) edge (m-2-3)
        (m-2-3) edge node [left] {\footnotesize $-$ }   (m-3-3)
        (m-3-3) edge node [left] {\footnotesize $-$ }   (m-4-3)
        (m-4-3) edge node [left] {\footnotesize $-$ }   (m-5-3)
        (m-5-3) edge node [left] {\footnotesize $-$ }   (m-6-3)
        (m-1-4) edge (m-2-4)
        (m-2-4) edge (m-3-4)
        (m-3-4) edge (m-4-4)
        (m-4-4) edge (m-5-4)
        (m-5-4) edge (m-6-4)
        (m-6-4) edge (m-7-4) 
        (m-6-3) edge (m-7-3);
        
\draw [->, rounded corners] (m-2-4) -- node [above] {\footnotesize $-$ }  +(2,0) -- +(2,-0.3) -- +(0,-0.3) -- +(-9.55,-0.3) -- +(-9.55,-2.5) -- +(-8.25,-2.5);
\draw [->, rounded corners] (m-3-4) --  +(2,0) -- +(2,-0.3) -- +(0,-0.3) -- +(-9.75,-0.3) -- +(-9.75,-2.5) -- +(-8.25,-2.5);
\draw [-, rounded corners] (m-4-4) -- +(2,0) -- +(2,-0.3) -- +(1,-0.3);
\draw [dotted,thick] (m-4-4) +(1,-0.3) -- +(0.5,-0.3);
\draw [-, rounded corners] (m-5-4) -- +(2,0) -- +(2,-0.3) -- +(1,-0.3);
\draw [-, rounded corners] (m-6-4) -- node [above] {\footnotesize $-$ }  +(2,0) -- +(2,-0.3) -- +(1,-0.3);
\draw [dotted,thick] (m-5-4) +(1,-0.3) -- +(0.5,-0.3);
\draw [dotted,thick] (m-6-4) +(1,-0.3) -- +(0.5,-0.3);
\end{tikzpicture}
\end{center}
The minus signs, indicating that negatives of the canonical maps are used, ensure that each square of the diagram anti-commutes. Consider the total complex, obtained by summing over the diagonal. Its first few terms are
\[
0 \rightarrow \Hom(S(U),U) \rightarrow \begin{array}{c} \Hom(S(U),W) \\ \oplus \\ \Hom(S(W),U) \end{array} \rightarrow \begin{array}{c} \Hom(S(U),V) \\ \oplus \\ \Hom(S(W),W) \\ \oplus \\ \Hom(S(V),U)  \end{array}  \rightarrow \cdots
\]
There is an action of $\mathbb{Z}_2$ on the total complex, commuting with all differentials, defined by letting the generator act by $(-1)^iS$ on $\Ext^i(S(W),W)$ and by $(-1)^{i+1}S$ on the remaining $i$th extension groups. Viewed as a virtual representation of $\mathbb{Z}_2$, the character of the total complex is zero. This implies the following relation between virtual dimensions of $\mathbb{Z}_2$-invariants,
\[
0 = \mathcal{E}(U) - \langle S(U), W \rangle  + \langle S(U),V \rangle + \left( \langle S(W),W \rangle - \mathcal{E}(W) \right) - \langle S(W),V \rangle + \mathcal{E}(V),
\]
which can be rewritten as 
\begin{equation}
\label{eq:kIden}
\mathcal{E}(W) = \mathcal{E}(U) + \mathcal{E}(V) + \langle S(U),V \rangle.
\end{equation}
The right-hand side of this equation is equal to $\mathcal{E}(U \oplus V)$, proving that $\mathcal{E}$ descends to $K(\mathcal{A})$.

Finally, it is straightforward to verify equation \eqref{eq:twistCocycle} using equation \eqref{eq:kIden}.
\end{proof}

\begin{Def}
With $c$ and $\tilde{c}$ as in Theorem \ref{thm:descendtoGroth}, $\mathcal{M}_{\mathcal{A}}$ is called the Ringel-Hall module.
\end{Def}

We have the following analogue of Green's bilinear form.

\begin{Lem}
\label{lem:sdGreenForm}
The $R$-valued symmetric bilinear form on $\mathcal{M}_{\mathcal{A}}$ defined by
\[
( [M], [N] )_{\mathcal{M}} = \frac{\delta_{M,N}}{a_S(M)}
\]
is non-degenerate and satisfies
\[
( x \otimes \xi , \rho( \zeta) )_{\mathcal{H} \otimes \mathcal{M}}  =  ( x \star \xi , \zeta )_{\mathcal{M}}, \;\;\;\ x \in \mathcal{H}_{\mathcal{A}}, \; \xi, \zeta \in \mathcal{M}_{\mathcal{A}}.
\]
\end{Lem}

We now give two examples.

Given an exact category $\mathcal{A}$, the triple $(H\mathcal{A}, S_H, 1_{H \mathcal{A}})$, with $H\mathcal{A} = \mathcal{A} \times \mathcal{A}^{op}$ and $S_H(A, B) = (B, A)$, is called the hyperbolic exact category with duality; all of its self-dual objects are hyperbolic. Let $\mathcal{H}_{\mathcal{A}}^{op-cop}$ be the (co)algebra obtained from $\mathcal{H}_{\mathcal{A}}$ by taking the opposite (co)multiplication. Then $\mathcal{H}_{H\mathcal{A}} \simeq \mathcal{H}_{\mathcal{A}} \otimes_R \mathcal{H}_{\mathcal{A}}^{op-cop}$. Note that $\mathcal{H}_{\mathcal{A}}$ is in a canonical way a left $\mathcal{H}_{\mathcal{A}} \otimes_R \mathcal{H}_{\mathcal{A}}^{op-cop}$-(co)module.

\begin{Prop}
\label{prop:disjUnion}
The map
\[
\mathcal{H}_{\mathcal{A}} \rightarrow \mathcal{M}_{H\mathcal{A}} , \;\;\;\; [X] \mapsto [H(X)]
\]
extends to an isomorphism of $\mathcal{H}_{\mathcal{A}} \otimes_R \mathcal{H}_{\mathcal{A}}^{op-cop}$-(co)modules preserving Grothendieck-Witt gradings and Green forms.
\end{Prop}
\begin{proof}
For simplicity set $c=\tilde{c}=0$. That the above map is an $R$-module isomorphism follows from the fact that a self-dual object of $H\mathcal{A}$ can be written as $U \oplus S_H(U)$ for a unique $U \in \mathcal{A}$. A subobject of $H(X)$, $X \in \mathcal{A}$, is necessarily of the form $U_1 \oplus S_H(U_2)$ for some $U_1, U_2 \in \mathcal{A}$, and is isotropic if and only if $S_H(U_2)  \subset  S_H(X \slash U_1)$. Summing over isomorphism types of $X \slash U_1$ shows
\[
G^{H(X)}_{U_1 \oplus S_H(U_2), H(Y)} = \sum_{W \in Iso(\mathcal{A})} F^X_{U_1, W} F^W_{Y, U_2},
\]
where we have used the equality 
\[
F^{S_H(W)}_{S_H(U_2), S_H(Y)} = F^{W}_{Y, U_2}.
\]
This implies that $G^{H(X)}_{U_1 \oplus S_H(U_2), H(Y)}$ is the coefficient of $[X]$ in $[U_1] [Y][U_2]$, all multiplication in $\mathcal{H}_{\mathcal{A}}$, establishing the $\mathcal{H}_{\mathcal{A}} \otimes_R \mathcal{H}_{\mathcal{A}}^{op}$-module isomorphism $\mathcal{H}_{\mathcal{A}} \simeq \mathcal{M}_{H\mathcal{A}} $.  Using
\[
Aut(U_1 \oplus S_H(U_2)) \simeq Aut(U_1) \times Aut (U_2), \;\;\; Aut_S(H(X)) \simeq Aut(X),
\]
similar arguments give the comodule isomorphism and show that Green forms are preserved. That the gradings are respected follows from the fact that the restriction of the hyperbolic functor to $\mathcal{A} \subset H \mathcal{A}$ induces an isomorphism $K(\mathcal{A}) \xrightarrow[]{\sim} GW(H\mathcal{A})$.
\end{proof}

\begin{Ex}
Let $\mathcal{H}_{Vect_X}$ and $\mathcal{M}_{Vect_X}$ be the Hall algebra and module associated to a smooth projective curve $X$ over $\mathbb{F}_q$, with duality determined by a line bundle $\mathcal{L}$ and a sign $s$. Following the automorphic interpretation of $\mathcal{H}_{Vect_X}$, $\mathcal{M}_{Vect_X}$ is identified with the space of $\mathcal{L}$-twisted unramified automorphic forms for symplectic or orthogonal groups over $\mathbb{F}_q(X)$. The Witt group of $(Vect_X, \mathcal{L}, \Theta)$ is finite \cite{arason1994} and therefore provides a finite $\mathcal{H}_{Vect_X}$-module decomposition of $\mathcal{M}_{Vect_X}$. As the duality does not extend to $Coh_X$ there is no obvious Hall module interpretation of Hecke operators on $\mathcal{M}_{Vect_X}$.
\end{Ex}

\subsection{An identity for self-dual Hall numbers}
In this section we prove the following theorem.

\begin{Thm}
\label{thm:sdRiedtmann}
Let $\mathcal{A}$ be a $\mathbb{F}_q$-linear hereditary finitary abelian category with duality. For all $U \in \mathcal{A}$ and $M \in \mathcal{A}_S$, the following identity holds:
\[
\sum_{N \in Iso(\mathcal{A}_S)} \frac{ \mathcal{G}^N_{U,M}}{a_S(N)} =q^{-\langle M, U \rangle - \mathcal{E}(U)}.
\]
\end{Thm}

We proceed in steps. Fix $\mathfrak{s} = (E; i, j, k, \pi) \in \underline{\mathcal{G}}^N_{U,M}$. Applying $\Hom(S(U), -)$ to the exact sequence $\mathfrak{s}_-=(j, \pi)$ gives the long exact sequence
\[
\cdots \rightarrow \Hom(S(U),M) \xrightarrow{\delta_-} \Ext^1(S(U),U) \xrightarrow[]{j_*} \Ext^1(S(U),E) \rightarrow \cdots
\]
Similarly, applying $\Hom(-, U)$ to $\mathfrak{s}_{\vert}=(S(\pi) \psi_M, S(j))$ gives
\[
\cdots \rightarrow \Hom(M,U) \xrightarrow{\delta_{\vert}} \Ext^1(S(U),U) \xrightarrow[]{S(j)^*} \Ext^1(S(E),U) \rightarrow \cdots
\]
Define $\delta_{\mathfrak{s}}^S: \Hom(M,U)  \rightarrow  \Ext^1(S(U),U)^S$ by
\[
\beta \mapsto \delta_{\vert} \beta + \delta_- (\psi_M^{-1} S(\beta)).
\]
Note that $\delta^S_{\mathfrak{s}}$ depends only on the class $[\mathfrak{s}_-]  \in \Ext^1(M,U)$.

Define the set of self-dual extensions of $M$ by $U$ by
\[
\leftexp{S}{\Ext}^1(M,U) = \bigsqcup_{N \in Iso(\mathcal{A}_S)} \underline{\mathcal{G}}^N_{U,M} \slash Aut_S(N),
\]
where $\phi \in Aut_S(N)$ acts by $\phi \cdot (E; i, j, k, \pi) = (E; \phi i, j, \phi k, \pi)$. The assignment $\mathfrak{s} \mapsto [\mathfrak{s}_-]$ defines maps $T: \underline{\mathcal{G}}^N_{U,M} \rightarrow \Ext^1(M,U)$ and $\tilde{T}: \leftexp{S}{\Ext}^1(M,U) \rightarrow \Ext^1(M,U)$.

\begin{Lem}
\label{lem:stab}
If $\mathfrak{s} \in \underline{\mathcal{G}}^N_{U,M}$ satisfies $T(\mathfrak{s})=\xi$, then
\[
\vert \text{\textnormal{Stab}}_{Aut_S(N)} (\mathfrak{s}) \vert = \vert \ker \delta_{\xi}^S \vert \vert \Hom(S(U),U)^{-S} \vert.
\]
\end{Lem}

\begin{proof}
An element $\phi \in Aut_S(N)$ fixes $\mathfrak{s}$ if and only if there exists $r \in Aut(E)$ with
\[
j =rj, \;\;\; \pi = \pi r^{-1},  \;\;\; \phi k = k r^{-1}.
\]
These equations imply that
\[
r =r_{\beta}=1_E - j \beta \pi, \;\;\;\;\;\;  \phi = \phi_{\tau}= 1_N + k \tau S(k) \psi_N
\]
for unique $\beta \in \Hom(M,U)$ and $\tau \in \Hom(S(E),E)$. The condition that $\phi_{\tau}$ is an isometry is equivalent to 
\[
\Theta_E^{-1} S(\tau) + \tau + \Theta_E^{-1} S(\pi \tau) \psi_M \pi \tau =0.
\]
Using this equation, it is straightforward to show that the above data fit into the commutative diagram
\[
\begin{tikzpicture}[every node/.style={on grid},  baseline=(current bounding box.center),node distance=1.8]
  \node (A) {$M$};
  \node (B) [right=of A]{$S(E)$}; 
  \node (C) [right=of B]{$S(U)$};
  \node (D) [below=of A]{$U$};
  \node (E) [right=of D]{$E$}; 
  \node (F) [right=of E]{$M$};
  \node (X) [left=of A]{$0$};
  \node (Y) [left=of D]{$0$};
  \node (Z) [right=of C]{$0$};
  \node (W) [right=of F]{$0$};
  \draw[->] (A)-- node [above] {\footnotesize $S(\pi) \psi_M$ }(B);
  \draw[->] (B)-- node [above] {\footnotesize $S(j)$ }  (C);
  \draw[->] (D)-- node [below] {\footnotesize $j$ }  (E);
  \draw[->] (E)-- node [below] {\footnotesize $\pi$ }  (F);
  \draw[->] (A)-- node [left] {\footnotesize $\beta$ }  (D);
  \draw[->] (C)-- node [right] {\footnotesize $-\psi_M^{-1} S(\beta)$ }  (F);
  \draw[->] (B)-- node [left] {\footnotesize $\tau$ }  (E);  
  \draw[->] (X)--  (A);  
  \draw[->] (Y)--  (D);  
  \draw[->] (C)--  (Z);  
  \draw[->] (F)--  (W);  
\end{tikzpicture}
\]

Conversely, for fixed $\beta$, the existence of a lift $\tau$, as in the above diagram, is equivalent to the condition $\beta \in \ker \, \delta_{\xi}^S$. Moreover, if $\beta \in \ker \, \delta_{\xi}^S$, then the set of its lifts is a $\Hom(S(U),U)$-torsor. Fix a lift $\tau_0$. Then $\Theta_E^{-1} S(\tau_0) + \tau_0 = j \mu S(j)$ for a unique $\mu \in \Hom(S(U),U)^S$. Putting
\[
\tau_1 = \tau_0 - \frac{1}{2}j \left(\mu +  \beta \psi_M^{-1} S(\beta) \right) S(j)
\]
gives $\phi_{\tau_1} \in \mbox{Stab}_{Aut_S(N)} (\mathfrak{s})$ lifting $\beta$. Finally, the set of lifts of $\beta$ to $\mbox{Stab}_{Aut_S(N)} (\mathfrak{s})$ is a torsor for the subgroup $\Hom(S(U),U)^{-S} \subset \Hom(S(U),U)$.
\end{proof}

\begin{Rem}
The ordinary version of Theorem \ref{thm:sdRiedtmann} is a corollary of the formula \cite{riedtmann1994}
\[
\frac{\mathcal{F}^X_{U,V} }{a(X)} = \frac{\vert \Ext^1(V,U)_X \vert}{\vert \Hom(V,U) \vert},
\]
where $\Ext^1(V,U)_X \subset \Ext^1(V,U)$ are the extensions with middle term isomorphic to $X$. This formula follows from the fact that all elements of $\mathcal{F}^X_{U,V}$ have $Aut(X)$-stabilizer isomorphic to $\Hom(V,U)$. Theorem \ref{thm:sdRiedtmann} is complicated by the fact that $\text{\textnormal{Stab}}_{Aut_S(N)} (\mathfrak{s})$ depends on more data than just $U$ and $M$.
\end{Rem}

Turning to the problem of describing the fibres of $\tilde{T}$, suppose that
\[
\mathfrak{t}_- : \;\;\;\;\;\;   U \overset{j}{\rightarrowtail} E \overset{\pi}{\twoheadrightarrow} M
\]
represents a class $\xi \in \Ext^1(M,U)$. Since $\mathcal{A}$ is hereditary, there exists an exact commutative diagram $\mathfrak{t}$ extending $\mathfrak{t}_-$:
\[
\begin{tikzpicture}[every node/.style={on grid},  baseline=(current bounding box.center),node distance=1.45]
  \node (A) {$U$};
  \node (B) [right=of A]{$E$}; 
  \node (C) [right=of B]{$M$};
  \node (D) [below=of A]{$U$};
  \node (E) [right=of D]{$N$}; 
  \node (F) [right=of E]{$S(E)$};
  \node (G) [below=of E]{$S(U)$};
  \node (H) [right=of G]{$S(U)$};
  
  \draw[>->] (A)-- node [above] {\footnotesize $j$ }(B);
  \draw[->>] (B)-- node [above] {\footnotesize $\pi$ }  (C);
  \draw[>->] (D)-- node [above] {\footnotesize $i$ }  (E);
  \draw[->>] (E)-- node [above] {\footnotesize $\rho$ }  (F);
  \draw[double equal sign distance] (A)-- (D);
  \draw[double equal sign distance] (G)-- (H);
  \draw[>->] (B)-- node [right] {\footnotesize $k$ }  (E);
  \draw[>->] (C)-- node [right] {\footnotesize $S(\pi) \psi_M$ }  (F);
  \draw[->>] (E)-- node [right] {\footnotesize $\tau$ }  (G);
  \draw[->>] (F)-- node [right] {\footnotesize $S(j)$ }  (H);
\end{tikzpicture}
\]
In \cite{grothendieck1972} (see also \cite{bertrand2013}) such a diagram is called an \textit{extension panach\'{e}e} of $S(E)$ by $E$. After using $\Theta$, the dual $S(\mathfrak{t})$ is another such extension panach\'{e}e. By \cite[\S 9.3.8.b]{grothendieck1972} there exists a unique $\gamma_{\mathfrak{t}} \in \Ext^1(S(U),U)$ such that $S(\mathfrak{t})$ and $\mathfrak{t} \sbt \gamma_{\mathfrak{t}}$ are isomorphic extensions panach\'{e}es. Here $\sbt$ denotes the simply transitive action of $\Ext^1(S(U),U)$ on isomorphism classes of extensions panach\'{e}es of $S(E)$ by $E$. Precisely, $\mathfrak{t}\sbt \gamma_{\mathfrak{t}}$ is the canonical lift of the Baer sum $N + j_* \gamma_{\mathfrak{t}} \in \Ext^1(S(U),E)$ to an extension panach\'{e}e of $S(E)$ by $E$.

\begin{Lem}
\label{lem:reductionStruct}
There exists a self-dual structure on $N$ satisfying $\rho = S(k) \psi_N$ (i.e. $\mathfrak{t} \in \underline{\mathcal{G}}^N_{U,M}$) if and only if $\gamma_{\mathfrak{t}} =0$. Moreover, if $\psi_N$ exists, then it is unique up to isometry.
\end{Lem}

\begin{proof}
The implication is clear. Conversely, if $\gamma_{\mathfrak{t}}=0$ then $\mathfrak{t} \simeq S(\mathfrak{t})$, that is, there exists an isomorphism $\psi: N \rightarrow S(N)$ satisfying $\psi k = S(\rho) \Theta_E$ and $S(k) \psi= \rho$. These equations imply that there exists a unique $\mu \in \Hom(S(U),U)^{-S}$ such that
\[
S(\psi) \Theta_N - \psi = S(\tau) \Theta_U \mu \tau.
\]
Then $\psi_N=\psi( 1_N + \frac{1}{2} i \mu \tau)$ is the desired self-dual structure.
\end{proof}

\begin{Lem}[see also {\cite[Lemme 3]{bertrand2013}}]
\label{lem:symmObstr}
The class $\gamma_{\mathfrak{t}}$ satisfies $\gamma_{\mathfrak{t}} + \Theta_{U *}^{-1} S(\gamma_{\mathfrak{t}}) =0$.
\end{Lem}
\begin{proof}
We claim that $S(\mathfrak{t} )\sbt \Theta_{U*}^{-1} S(\gamma_{\mathfrak{t}}) = \mathfrak{t}$. The lemma will then follow from the definition of $\gamma_{\mathfrak{t}}$ and the freeness of the $\Ext^1(S(U),U)$-action. According to \cite[Lemme A.1]{bertrand2013}, $S(\mathfrak{t}) \sbt \Theta_{U*}^{-1} S(\gamma_{\mathfrak{t}})$ can also be described as the canonical lift of
\[
S(N) + S(j)^* \Theta_{U*}^{-1} S(\gamma_{\mathfrak{t}}) \in \Ext^1(S(E),U)
\]
to an extension panach\'{e}e. Since this extension is also the class of the middle horizontal exact sequence of $S(\mathfrak{t} \sbt \gamma_{\mathfrak{t}})$, the claim follows.
\end{proof}

The proof of \cite[Th\'{e}or\`{e}me 1]{bertrand2013} shows
\[
\gamma_{\mathfrak{t} \sbt \lambda} = \gamma_{\mathfrak{t}} + \lambda - \Theta_{U*}^{-1} S(\lambda), \;\;\; \lambda \in \Ext^1(S(U),U).
\]
From Lemma \ref{lem:symmObstr}, we conclude that $\gamma_{\mathfrak{t} \sbt (-\frac{1}{2} \gamma_{\mathfrak{t}})} =0$ and, by Lemma \ref{lem:reductionStruct}, that the diagram $\mathfrak{t} \sbt (-\frac{1}{2} \gamma_{\mathfrak{t}})$ extends to an element of $\underline{\mathcal{G}}^N_{U,M}$. In particular, $\tilde{T}^{-1}(\xi)$ is non-empty.

\begin{Lem}
\label{lem:fibre}
The action of $\Ext^1(S(U),U)^S$ on $\tilde{T}^{-1}(\xi)$ is transitive with stabilizer $\text{\textnormal{im}}\, \delta_{\xi}^S$. In particular,
\[
\vert \tilde{T}^{-1}(\xi) \vert = \frac{\vert \Ext^1(S(U),U)^S \vert }{\vert \im \delta^S_{\xi} \vert}.
\]
\end{Lem}
\begin{proof}
Transitivity can be verified as in \cite[Th\'{e}or\`{e}m 1]{bertrand2013}. Consider the diagrams $\mathfrak{t}$ and
\[
\mathfrak{t} \sbt \delta_{\xi}^S \beta = (\mathfrak{t} \sbt \delta_- \beta) \sbt \delta_{\vert} \psi_M^{-1} S(\beta).
\]
By \cite[Lemme A.2]{bertrand2013}, $\mathfrak{t} \sbt \delta_- \beta$ is obtained from $\mathfrak{t}$ by replacing $k$ with $k r_{\beta}^{-1}$. Similarly, $\mathfrak{t} \sbt \delta_{\xi}^S \beta$ is obtained from $\mathfrak{t} \sbt \delta_- \beta$ by replacing $S(\rho)$ with $S(r_{\beta}^{-1}) S(\rho)$. Hence, $\mathfrak{t}$ and $\mathfrak{t} \sbt \delta_{\xi}^S \beta$ differ by an automorphism of $\mathfrak{t}_-$ and are therefore equal as self-dual exact sequences. So, $\im \delta_{\xi}^S$ acts trivially.

On the other hand, suppose that $\mathfrak{t}, \mathfrak{t}^{\prime} \in \mathcal{G}^N_{U,M}$ are equal in $\leftexp{S}{\Ext}^1(M,U)$. Without loss of generality, we can assume $\mathfrak{t}_-= \mathfrak{t}^{\prime}_-$. Then, by assumption, there exists $\phi \in Aut_S(N)$ satisfying $\phi k = k^{\prime} r^{-1}_{\beta}$ for some $\beta\in \Hom(M,U)$.  The discussion in the previous paragraph show that replacing $k^{\prime}$ with $k^{\prime} r^{-1}_{\beta}$ gives the same self-dual exact sequence but a different extension panach\'{e}e, namely $\mathfrak{t}^{\prime} \sbt \delta_{\xi}^S \beta$. The map $\phi$ is now an isomorphism of extensions panach\'{e}es $\mathfrak{t} \xrightarrow[]{\sim} \mathfrak{t}^{\prime} \sbt \delta_{\xi}^S \beta$. In particular, if $\mathfrak{t} = \mathfrak{t} \sbt \gamma$ in $\leftexp{S}{\Ext}^1(M,U)$, then $\gamma \in \im \delta_{\xi}^S$.
\end{proof}

\begin{proof}[Proof of Theorem \ref{thm:sdRiedtmann}]
We compute using Burnside's lemma
\begin{align*}
\sum_{N \in Iso(\mathcal{A}_S)} \frac{\mathcal{G}^N_{U,M} }{a_S(N)} &= \sum_{[\mathfrak{s}] \in \leftexp{S}{\Ext}^1(M,U)} \frac{1}{\vert \mbox{Stab}_{Aut_S(N)}(\mathfrak{s}) \vert} \\
&= \sum_{\xi \in \Ext^1(M,U)} \frac{\vert \tilde{T}^{-1}(\xi) \vert}{  \vert \ker \delta^S_{\xi} \vert \vert \Hom(S(U),U)^{-S} \vert }  \;\;\;\;\;\;\;\;\;\; \;\;\;\;\;\; (\hbox{Lemma \ref{lem:stab}} )\\
&= \sum_{\xi \in \Ext^1(M,U)} \frac{\vert \Ext^1(S(U),U)^S \vert}{ \vert \im \delta^S_{\xi} \vert  \vert \ker \delta^S_{\xi} \vert \vert \Hom(S(U),U)^{-S} \vert } \;\;\;\;\;  (\hbox{Lemma \ref{lem:fibre}})\\
&= \sum_{\xi \in \Ext^1(M,U)} \frac{\vert \Ext^1(S(U),U)^S \vert }{ \vert \Hom(M,U) \vert \vert \Hom(S(U),U)^{-S} \vert } \\
&= q^{-\langle M, U \rangle - \mathcal{E}(U)}.
\end{align*}
\end{proof}

\section{Hall modules from quivers with involution}
\label{sec:hallModQuantGrp}
For the remainder of the paper we assume that Hall algebras and modules are given the Ringel twist.

\subsection{Quantum groups and the Hall algebra of a quiver}
\label{subsec:quantGrpSum}
Let $A=(a_{ij})_{i,j=1}^n$ be a symmetric generalized Cartan matrix with associated derived Kac-Moody algebra $\mathfrak{g}$.  The root lattice $\Phi = \bigoplus_{i=1}^n \mathbb{Z} \epsilon_i$ is generated by the simple roots $\epsilon_1, \dots, \epsilon_n$. The Cartan form $(-, -) : \Phi \times \Phi \rightarrow \mathbb{Z}$ satisfies $(\epsilon_i, \epsilon_j) = a_{ij}$.

Let $\mathbb{Q}(v)$ be the field of rational functions in an indeterminate $v$. Define elements of the subring of Laurent polynomials $\mathbb{Z}[v,v^{-1}] \subset \mathbb{Q}(v)$ by
\[
[n]_v = \frac{v^n - v^{-n}}{v-v^{-1}}, \;\;\;\; [n]_v! = \prod_{i =1}^n [i]_v , \;\;\;\;\; \left[ \begin{array}{c} n \\ k \end{array} \right]_v = \frac{[n]_v! }{[k]_v! [n-k]_v!}, \;\;\;\; n, k \in \mathbb{Z}_{\geq 0}.
\]

\begin{Def}
The quantum Kac-Moody algebra $U_v(\mathfrak{g})$ is the $\mathbb{Q}(v)$-algebra generated by symbols $E_i, F_i, T_i, T_i^{-1}$, for $i =1, \dots, n$, subject to the relations
\begin{enumerate}[itemsep=1pt,parsep=1pt]
\item $[T_i ,T_j ]=0$ and $T_i T_i^{-1} = 1$ for $i=1, \dots, n$,

\item$T_i E_j  T_i^{-1} = v^{(\epsilon_i, \epsilon_j)} E_j$ and $T_i F_j  T_i^{-1} = v^{-(\epsilon_i, \epsilon_j)} F_j$ for $i,j=1, \dots, n$,

\item $[E_i, F_j] = \delta_{ij}\frac{T_i - T_i^{-1}}{v - v^{-1}}$ for $i,j=1, \dots, n$, and

\item (quantum Serre relations) for any $i,j =1, \dots, n$, with $i\neq j$, 
\[
\sum_{p=0}^a (-1)^p \left[ \begin{array}{c} a \\ p \end{array} \right]_v   F_i^p F_j F_i^{a-p} =0 , \;\;\;\;\;\;\;\;
\sum_{p=0}^a (-1)^p \left[ \begin{array}{c} a \\ p \end{array} \right]_v  E_i^p E_j E_i^{a-p} =0 
\]
where $a= 1- (\epsilon_i, \epsilon_j)$.
\end{enumerate}
\end{Def}

Let $U^-_v(\mathfrak{g})$ be the subalgebra of $U_v(\mathfrak{g})$ generated by $F_i$, $i=1, \dots, n$. For $\nu \in \mathbb{C}^{\times}$ not a root of unity, the specialized quantum groups $U_{\nu}(\mathfrak{g}), U^-_{\nu}(\mathfrak{g})$ are the $\mathbb{Q}[\nu, \nu^{-1}]$-algebras with generators and relations as above but with $v$ replaced with $\nu$.

We now recall the connection between quantum groups and Hall algebras of quivers. Consider a quiver $Q$ with finite sets of nodes $Q_0$ and arrows $Q_1$ and head and tail maps $h,t:  Q_1 \rightarrow Q_0$. A representation of $Q$ is a pair $(V, v)$ consisting of a finite dimensional $Q_0$-graded vector space $V= \bigoplus_{i \in Q_0} V_i$ and a collection $v = \{ v_{\alpha} \}_{\alpha \in Q_1}$ of linear maps $V_{t(\alpha)} \xrightarrow[]{v_{\alpha}} V_{h(\alpha)}$.  The category $Rep_k(Q)$ of $k$-representations of $Q$ is abelian and hereditary.  The abelian group $\mathbb{Z}^{Q_0}$ of virtual dimension vectors has a natural basis $\{\epsilon_i\}_{i \in Q_0}$ consisting of unit vectors supported at $i\in Q_0$. The simple representation with dimension vector $\epsilon_i$ and all structure maps zero is denoted by $S_i$.

When $Q$ has no loops its symmetrized Euler form in the basis $\{\epsilon_i \}_{i \in Q_0}$ is a symmetric generalized Cartan matrix. Denote by $\mathfrak{g}_Q$ the corresponding derived Kac-Moody algebra and let  $\mathcal{H}_Q$ be the Hall algebra of $Rep_{\mathbb{F}_q}(Q)$.

\begin{Thm}[\cite{ringel1990v3}, \cite{green1995}]
\label{thm:quantumHall}
If $Q$ has no loops, then the subalgebra of $\mathcal{H}_Q$ generated by $[S_i]$, for all $i \in Q_0$, is isomorphic to $U^-_{\nu}(\mathfrak{g}_Q)$.
\end{Thm}

\subsection{Representations of a quiver with involution}
\label{subsec:quivRep}

\begin{Def}
An involution of $Q$ is a pair of involutions $Q_i \xrightarrow[]{\sigma} Q_i$, $i=0,1$, such that for all $\alpha \in Q_1$, $h(\sigma(\alpha)) = \sigma(t(\alpha))$ and if $\sigma(t(\alpha)) = h(\alpha)$ then $\sigma(\alpha)=\alpha$.
\end{Def}

Not every quiver admits an involution. For example, a simply laced Dynkin quiver that admits an involution is necessarily of type $A$. On the other hand, the double of a quiver always admits an involution.

Let $(Q,\sigma)$ be a quiver with involution. To construct a duality on $Rep_k(Q)$, let $\iota$ be an involutive field automorphism of $k$ with fixed subfield $k_0$. Fix also functions $s: Q_0 \rightarrow \{ \pm 1 \}$ and $\tau: Q_1 \rightarrow \{ \pm 1 \}$ satisfying $s_i=s_{\sigma(i)}$ and $\tau_{\alpha} \tau_{\sigma(\alpha)} = s_i s_j$ for all $i \xrightarrow[]{\alpha} j$. The functor $S: Rep_k(Q) \rightarrow Rep_k(Q)$ is defined by setting $S(U,u)$ equal to
\[
S(U)_i =\overline{U}_{\sigma(i)}, \;\;\;\;\;\; S(u)_{\alpha} = \tau_{\alpha} u_{\sigma(\alpha)}^{\vee},
\]
where
\[
\overline{U}_{\sigma(i)} = \{ f \in \Hom_{k_0}(U_{\sigma(i)}, k) \; \vert \; f ( c v) = \iota(c) f(v), \;\; v \in U_{\sigma(i)}, \;\; c \in k \}.
\]
Given a morphism $U \xrightarrow[]{\phi} U^{\prime}$, the $i$th component of $S(U^{\prime}) \xrightarrow[]{S(\phi)} S(U)$ is $\phi^{\vee}_{\sigma(i)}$. Put $\Theta= \bigoplus_{i \in Q_0} s_i \cdot \overline{\mbox{ev}}_i$, where $\overline{\mbox{ev}}$ is the composition of the evaluation map with $\iota$. Then $(Rep_k(Q), S, \Theta)$ is a $k_0$-linear abelian category with duality.

Geometrically, a self-dual structure $\psi_M$  on $M \in Rep_k(Q)$ defines a non-degenerate form on $M$ by $\langle v, w \rangle = \psi_M(v)(w)$. This form is linear in the first variable and $\iota$-linear in the second variable. Moreover,  $M_i$ and $M_j$ are orthogonal unless $i = \sigma(j)$, in which case the restriction of the form to $M_i + M_{\sigma(i)}$ is $s_i$-symmetric (resp. $s_i$-hermitian) if $\iota$ is trivial (resp. non-trivial). Finally, the structure maps satisfy
\[
\langle m_{\alpha} v, w \rangle - \tau_{\alpha}  \langle v, m_{\sigma(\alpha)} w \rangle =0, \;\;\;\;\;\; v \in M_{t(\alpha)}, \; w \in M_{\sigma(h(\alpha))}.
\]

When $\iota$ is the identity, $\tau = -1$ and $s$ is constant, self-dual objects are the orthogonal and symplectic representations (referred to as the $s=1$ and $-1$ cases, respectively) introduced by Derksen and Weyman \cite{derksen2002}. For other choices of $(s, \tau)$ the self-dual objects were introduced in \cite{zubkov2005}.

When $\iota$ is non-trivial and $(s, \tau) = (1,-1)$, self-dual objects will be called unitary representations (referred to as the $s=0$ case). In particular, if $k= \mathbb{F}_q$ then $q$ must be a perfect square. We then regard $Rep_{\mathbb{F}_q}(Q)$ as a $\mathbb{F}_{\sqrt{q}}$-linear category and take $\nu_0=\sqrt[4]{q}^{-1}$ and $R=\mathbb{Q}[\nu_0,\nu_0^{-1}]$ in the definition of $\mathcal{H}_Q$. We also rescale $\mathcal{E}$ by a factor of $\frac{1}{2}$. While $\mathcal{E}$ is then only half-integral, the quantity $\nu^{-\mathcal{E}(U)}= \sqrt{q}^{\mathcal{E}(U)}$ remains integral.

\begin{Ex}
The quiver $\begin{tikzpicture}[thick,scale=.33,decoration={markings,mark=at position 0.6 with {\arrow{>}}}]
\draw[postaction={decorate}] (0,0) to  (2,0);
\fill (0,0) circle (4pt);
\fill (2,0) circle (4pt);
\end{tikzpicture}$ has a unique involution, swapping the nodes and fixing the arrow. An orthogonal representation is a skew-symmetric map $V \rightarrow V^{\vee}$. Isometry classes of orthogonal representations are parameterized by $\Lambda^2 k^n \slash GL_n$.
\end{Ex}

\begin{Ex}
The Jordan quiver $\hspace*{-10pt} \begin{tikzpicture}[thick, scale=.33,decoration={markings,mark=at position 0.55 with {\arrow{>}}}]
\path (0,0) edge[postaction={decorate},loop above, out=140, in=50,looseness=0.8, distance=2cm] (0,0);
\fill (0,0) circle (4pt);
\end{tikzpicture} \hspace*{-8pt}$ has a unique involution, the trivial involution. A symplectic representation consists of a symplectic vector space $M$ and $m \in \mathfrak{sp}(M)$. Isometry classes of symplectic representations are parametrized by the adjoint orbits of $Sp_{2n}$ on $\mathfrak{sp}_{2n}$.
\end{Ex}

\begin{Ex}
For any quiver $Q$, let $Q^{op}$ be the opposite quiver, obtained by reversing the orientations of all arrows of $Q$. Then $Q^{\sqcup} = Q \sqcup Q^{op}$ has an involution that sends a node (arrow) of $Q$ to the corresponding node (arrow) of $Q^{op}$. For any duality $(S, \Theta)$ on $Rep_{\mathbb{F}_q}(Q^{\sqcup})$, there are $\mathcal{H}_{Q} \otimes_R \mathcal{H}_Q^{op-cop}$-(co)module isomorphisms
\[
\mathcal{M}_{Q^{\sqcup}} \simeq \mathcal{M}_{H Rep(Q)} \simeq \mathcal{H}_Q.
\]
Indeed, the functor $H Rep_k(Q) \rightarrow (Rep_k(Q^{\sqcup}),S,\Theta)$ sending $(A,B)$ to $A \oplus S(B)$ together with the isomorphism $\left( \begin{smallmatrix}  0 & 1 \\ \Theta &0 \end{smallmatrix} \right) : F \circ S_H \rightarrow S \circ F$ define an equivalence of categories with duality. This gives the first isomorphism. The second follows from Proposition \ref{prop:disjUnion}.
\end{Ex}

The Grothendieck group $K(Rep_k(Q))$ is the free abelian group with basis the set $\mathfrak{S}$ of isomorphism classes of simple representations. Write $\mathfrak{S}= \mathfrak{S}^+ \sqcup \mathfrak{S}^S \sqcup \mathfrak{S}^-$ where $\mathfrak{S}^S$ consists of simples fixed by $S$ and $S(\mathfrak{S}^+)=\mathfrak{S}^-$.

\begin{Prop}
\label{prop:gwGroupQuivers}
There are canonical group isomorphisms
\[
GW(Rep_k(Q)) \simeq \mathbb{Z} \mathfrak{S}^+ \oplus  \bigoplus_{U \in \mathfrak{S}^S} GW(\mathcal{A}_U), \;\;\;\;\;\;\; W(Rep_k(Q)) \simeq \bigoplus_{U \in \mathfrak{S}^S} W(\mathcal{A}_U)
\]
where $\mathcal{A}_U$ is the semisimple abelian subcategory with duality generated by $U$.
\end{Prop}
\begin{proof}
Let $i: U \rightarrowtail N$ be a simple subrepresentation. By Schur's lemma, $S(i) \psi_N i$ is zero or an isomorphism. In the former case $U$ is isotropic and $\vert N \vert = \vert H(U) \vert + \vert N /\!/ U \vert$ in $GW(Rep_k(Q))$. In the latter case $S(i) \psi_N i$ is a self-dual structure on $U$. Therefore $N \simeq_S U \oplus \tilde{N}$, which implies $\vert N \vert = \vert U \vert + \vert \tilde{N} \vert$. As every representation has a finite composition series we can repeatedly apply the above procedure, giving the description of $GW(Rep_k(Q))$. The description of $W(Rep_k(Q))$ now follows from the exact sequence \eqref{eq:GWexact}.
\end{proof}

If $Q$ is acyclic then $\mathfrak{S}=\{S_i \}_{i \in Q_0}$ and $\mathcal{A}_U \simeq Vect_k$ with duality determined by $s$. When $k= \mathbb{F}_q$, $GW(Vect_{\mathbb{F}_q})$ is isomorphic to $\mathbb{Z}$ (resp. $\mathbb{Z}^2$) if $s\in \{-1,0\}$ (resp. $s=1$) and
\[
W (Vect_{\mathbb{F}_q}) \simeq \left\{ \begin{array}{ll} \{1 \}, & \hbox{ if } s=-1 \\ \mathbb{Z}_2, & \hbox{ if } s=0 \\ \mathbb{Z}_4, & \hbox{ if } s=1 \hbox{ and } q \equiv 3 \; (\hbox{mod} \, 4)  \\   \mathbb{Z}_2 \times \mathbb{Z}_2, & \hbox{ if } s=1 \hbox{ and } q \equiv 1 \; (\hbox{mod} \, 4).  \end{array} \right.
\]
In particular, for $s=1$, this is the classical Witt group $W(\mathbb{F}_q)$ of orthogonal forms. Note that the Grothendieck-Witt class of a representation is essentially its dimension vector, with additional decorations at $\sigma$-fixed vertices with $s_i=1$.

We now determine explicitly the function $\mathcal{E}$ for $Rep_k(Q)$. Given representations $V$ and $W$ define
\[
A^0(V,W) = \bigoplus_{i \in Q_0} \hbox{Hom}_k(V_i,W_i), \;\;\;\;\; A^1(V,W) = \bigoplus_{i \xrightarrow[]{\alpha} j \in Q_1} \Hom_k(V_i,W_j).
\]
There is a differential $A^0(V,W) \xrightarrow[]{\delta} A^1(V,W)$ given by $\delta \{f_i \}_i = \{ w_{\alpha} f_i - f_j v_{\alpha} \}_{\alpha}$. The cohomology of $\delta$ fits into the exact sequence (e.g. \cite{crawley1992})
\begin{equation}
\label{eq:ordRes}
0 \rightarrow \Hom(V,W) \rightarrow A^0(V,W) \xrightarrow[]{\delta} A^1(V,W) \rightarrow \Ext^1(V,W) \rightarrow 0.
\end{equation}
It follows that the Euler form depends only on the dimension vectors of its arguments:
\[
\langle d, d^{\prime} \rangle = \sum_{i \in Q_0} d_i d^{\prime}_i - \sum_{i \xrightarrow[]{\alpha} j } d_i d^{\prime}_j, \;\;\;\;\; d, d^{\prime} \in \mathbb{Z}^{Q_0}.
\]

\begin{Prop}
\label{prop:explicitEulerMod}
Let $U \in Rep_k(Q)$. Then
\[
\mathcal{E}(U)=\sum_{i \in Q_0^{\sigma}} \frac{u_i(u_i -s_i)}{2}  + \sum_{i \in Q_0^+} u_{\sigma(i)} u_i - \sum_{\sigma(i) \xrightarrow[]{\alpha} i \in Q_1^{\sigma}} \frac{u_i(u_i + \tau_{\alpha} s_i)}{2} -\sum_{i \xrightarrow[]{\alpha} j  \in Q_1^+} u_{\sigma(i)} u_j
\]
where $u=\mathbf{dim}\, U$. Here $Q_0= Q_0^+ \sqcup Q_0^{\sigma} \sqcup Q_0^+$, where $Q_0^{\sigma}$ consists of the $\sigma$-fixed vertices and $\sigma(Q_0^+) = Q_0^-$. The decomposition of $Q_1$ is analogous.
\end{Prop}
\begin{proof}
The involution of $A^{\sbt}(S(U),U)$ defined by the composition
\[
A^i(S(U),U) \xrightarrow[]{S} A^i(S(U),S^2(U)) \xrightarrow[]{\Theta_{U *}^{-1}} A^i(S(U),U)
\]
anticommutes with $\delta$. The cohomology of the subcomplex of (anti-)fixed points fits into the exact sequence
\[
0 \rightarrow \Hom(S(U),U)^{-S} \rightarrow A^0(S(U),U)^{- S} \xrightarrow[]{\delta} A^1(S(U),U)^S \rightarrow \Ext^1(S(U),U)^S \rightarrow 0.
\]
The Euler characteristic of this sequence, after some routine calculation, gives the claimed formula for $\mathcal{E}$.
\end{proof}

\subsection{\texorpdfstring{$B_{\sigma}(\mathfrak{g}_Q)$}-module structure of \texorpdfstring{$\mathcal{M}_Q$}.}

In this section we assume that $Q$ has no loops. Theorems \ref{thm:hallModuleExact} and \ref{thm:quantumHall} imply that $\mathcal{M}_Q$ is a representation of $U^-_{\nu}(\mathfrak{g}_Q)$. Since $\mathcal{H}_Q$ itself is a quantum Borcherds algebra \cite[Theorem 1.1]{sevenhant2001}, $\mathcal{M}_Q$ is also a representation of a much larger quantum group. However, without a better understanding of the full structure of $\mathcal{H}_Q$ it is difficult to use this to say much about $\mathcal{M}_Q$. Instead, we focus on incorporating the comodule structure. The na\"{i}ve guess that $\mathcal{M}_Q$ is a Hopf module, possibly after a twist as in Theorem \ref{thm:green}, is already false for the quiver consisting of a single node and no arrows. We therefore seek a replacement of the Hopf module condition.

To begin, we recall a modification of Kashiwara's $q$-boson algebra. With the notation of Section \ref{subsec:quantGrpSum}, suppose we are also given an involution $\sigma$ of the set of simple roots of $\mathfrak{g}$ that preserves the Cartan form.

\begin{Def}[\cite{enomoto2008}]
The reduced $\sigma$-analogue $B_{\sigma}(\mathfrak{g})$ is the $\mathbb{Q}(v)$-algebra generated by symbols $E_i, F_i,T_i, T_i^{-1}$, for  $i=1, \dots, n$, subject to the relations
\begin{enumerate}[itemsep=1pt,parsep=1pt]
\item $[T_i, T_j]=0$, $T_i T_i^{-1} =1$ and $T_i = T_{\sigma(i)}$ for $i=1, \dots, n$,

\item $T_i E_j = v^{(\epsilon_j + \epsilon_{\sigma(j)} , \epsilon_i)} E_j T_i$ and $T_i F_j = v^{-(\epsilon_j + \epsilon_{\sigma(j)} , \epsilon_i)} F_j T_i$ for $i,j=1, \dots, n$,

\item $E_i F_j = v^{-(\epsilon_i, \epsilon_j)} F_j E_i + \delta_{i,j} + \delta_{i, \sigma(j)}T_i$ for $i,j=1, \dots, n$, and

\item the quantum Serre relations for $E_i$ and $F_i$.
\end{enumerate}
\end{Def}

We will use the following characterization of highest-weight $B_{\sigma}(\mathfrak{g})$-modules.

\begin{Prop}[{\cite[Proposition 2.11]{enomoto2008}}]
\label{prop:Vsigma}
Let $\lambda \in \Hom(\Phi, \mathbb{Z})$ be a $\sigma$-invariant integral weight of $\mathfrak{g}$. Then there exists a $B_{\sigma}(\mathfrak{g})$-module $V_{\sigma}(\lambda)$ generated by a non-zero vector $\phi_{\lambda}$ such that $T_i \phi_{\lambda} = v^{\lambda(\epsilon_i)} \phi_{\lambda}$ for all $i =1, \dots, n$ and
\[
\left\{ x \in V_{\sigma}(\lambda) \; \vert \; E_i x =0, \;\;  i=1, \dots, n \right\} = \mathbb{Q}(v) \phi_{\lambda}.
\]
Moreover, $V_{\sigma}(\lambda)$ is irreducible and is unique up to isomorphism.
\end{Prop}

We require two straightforward variations of Proposition \ref{prop:Vsigma}. The first is the extension to $\sigma$-invariant half-integral weights $\lambda \in \Hom(\Phi, \frac{1}{2}\mathbb{Z})$, in which case $V_{\sigma}(\lambda)$ is a $B_{\sigma}(\mathfrak{g}) \otimes_{\mathbb{Q}(v)} \mathbb{Q}(v^{\frac{1}{2}})$-module. The second is an extension to representations of generic specializations $B_{\sigma}(\mathfrak{g})_{\nu}$, the resulting modules written $V_{\sigma}(\lambda)_{\nu}$. The proof in \cite{enomoto2008}  carries over directly in both cases.

Returning to Hall modules, define operators $E_i, F_i, T_i \in \mbox{End}_R(\mathcal{M}_Q)$ as follows (see also \cite{enomoto2009}). Put
\[
F_i [M] = [S_i] \star [M] = \nu^{-\langle M, S_i \rangle - \mathcal{E}(S_i)} \sum_N  G_{S_i,M}^N [N]
\]
and let $E_i$ be the projection of $\rho$ onto $[S_i] \otimes \mathcal{M}_Q \subset \mathcal{H}_Q \otimes \mathcal{M}_Q$:
\[
E_i [N] = \sum_M \nu^{-\langle M, S_i \rangle - \mathcal{E}(S_i)} \frac{a(S_i) a_S(M)}{a_S(N)} G_{S_i,M}^N [M].
\]
Finally,
\[
T_i [M] = \nu^{-( \mathbf{dim} \, M, \epsilon_i) - \mathcal{E}(\epsilon_i) - \mathcal{E}(\epsilon_{\sigma(i)})}[M].
\]

Abusing notation slightly, let
\[
B_{\sigma}(\mathfrak{g}_Q)_{\nu_0} = B_{\sigma}(\mathfrak{g}_Q)_{\nu} \otimes_{\mathbb{Q}[\nu,\nu^{-1}]} \mathbb{Q}[\nu_0, \nu_0^{-1}].
\]
If $\nu_0= \nu$ there is no conflict of notation, but if $\nu_0 = \sqrt{\nu}$, then $B_{\sigma}(\mathfrak{g}_Q)_{\nu_0}$ is different from the specialization of $B_{\sigma}(\mathfrak{g}_Q)$ to $\nu_0$.

\begin{Thm}
\label{thm:redModuleStructure}
The operators $E_i, F_i, T_i$, for $i\in Q_0$, give $\mathcal{M}_Q$ the structure of a $B_{\sigma}(\mathfrak{g}_Q)_{\nu_0}$-module.
\end{Thm}
\begin{proof}[Beginning of proof]
The first two parts of the first relation satisfied by $B_{\sigma}(\mathfrak{g}_Q)$ are clear while $T_i = T_{\sigma(i)}$ because $(d, \epsilon_i) = (d, \epsilon_{\sigma(i)})$ for all $\sigma$-symmetric $d \in \mathbb{Z}^{Q_0}$.  The second relation follows from the fact that $F_i$ (resp. $E_i$) increases (resp. decreases) the dimension vector by $\epsilon_i + \epsilon_{\sigma(i)}$. The  quantum Serre relations for $F_i$ follow from Theorems \ref{thm:hallModuleExact} and \ref{thm:quantumHall}. Using Lemma \ref{lem:sdGreenForm} we find
\[
( F_i \xi, \zeta )_{\mathcal{M}} = \frac{1}{\nu^{-2}-1} ( \xi, E_i \zeta )_{\mathcal{M}}, \;\;\;\; \xi, \zeta \in \mathcal{M}_Q.
\]
The quantum Serre relations for $E_i$ then follow from those for $F_i$ and the non-degeneracy of $( \cdot, \cdot )_{\mathcal{M}}$. To complete the proof it remains to verify the third relation, whose proof we break into a number of parts.
\end{proof}

Using Lemma \ref{lem:diffHallNums}, the third relation is seen to be equivalent to the following identity, for all $i, j \in Q_0$ and self-dual representations $X,Y$:
\begin{align}
\label{eq:sdHallIden}
\sum_N \frac{\mathcal{G}^N_{S_i,X} \mathcal{G}^N_{S_j, Y}}{a_S(N)} &= \frac{ \vert \Ext^1(S_{\sigma(j)}, S_i ) \vert }{\vert \Hom(S_{\sigma(j)}, S_i) \vert}  \sum_Z \frac{ \mathcal{G}^Y_{S_i,Z} \mathcal{G}^X_{S_j,Z} }{a_S(Z)} + \delta_{i, \sigma(j)}\delta_{X,Y} a(S_i) a_S(X)\nonumber\\
 &\qquad + \delta_{i,j} \delta_{X,Y} a(S_i) a_S(X) \frac{\vert \Ext^1(X, S_i) \vert \vert \Ext^1(S_{\sigma(i)}, S_i)^S \vert}{\vert \Hom(X,S_i)\vert \vert \Hom(S_{\sigma(i)},S_i)^{-S} \vert} .
\end{align}
We will complete the proof of Theorem \ref{thm:redModuleStructure} by proving this identity.

\begin{figure}
\begin{center}
\begin{tikzpicture}[every node/.style={on grid},  baseline=(current bounding box.center),node distance=1.3]
  \node (A) {$V$};
  \node (B) [below=of A]{$N$};
  \node (C) [below=of B]{$Y$};
  \node (D) [left=of B]{$U$};
  \node (E) [right=of B]{$X$};
  \draw[>->] (A)-- node [left] {\footnotesize $i_V$ }  (B);
  \draw[dashed,->>] (B)-- node [left] {\footnotesize $\pi_V$ }  (C);
  \draw[>->] (D)-- node [above] {\footnotesize $i_U$ }  (B);
  \draw[dashed,->>] (B)-- node [above] {\footnotesize $\pi_U$ }  (E);
\end{tikzpicture}
\hspace{70pt}
\begin{tikzpicture}[every node/.style={on grid},  baseline=(current bounding box.center),node distance=1.3]
  \node (A) {$V$};
  \node (B) [below=of A]{$X$};
  \node (C) [below=of B]{$Z$};
  \node (D) [left=of C]{$Y$};
  \node (E) [left=of D]{$U$};
 \draw[>->] (A)-- node [right] {\footnotesize $\tilde{i}_V$ }  (B);
  \draw[dashed,->>] (B)-- node [right] {\footnotesize $\tilde{\pi}_V$ }  (C);
  \draw[dashed,->>] (D)-- node [above] {\footnotesize $\tilde{\pi}_U$ }  (C);
  \draw[>->] (E)-- node [above] {\footnotesize $\tilde{i}_U$ }  (D);
\end{tikzpicture}
\caption{ (Left) a cross diagram; (right) a corner diagram.}
\label{fig:crossCorner}
\end{center}
\end{figure}
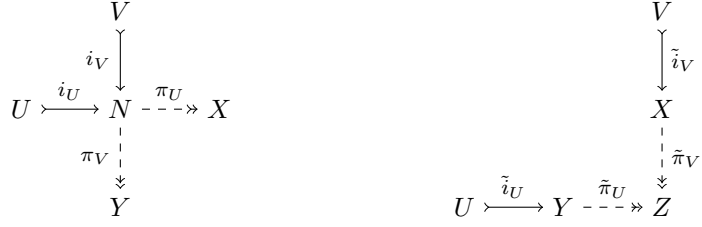

Given representations $U,V$ and self-dual representations $N,X,Y$, the set of crosses of self-dual exact sequences, as in Figure \ref{fig:crossCorner}, is written $C_N(U,V; X,Y)$. The group $Aut_S(N)$ acts on $C_N(U,V;X,Y)$ with orbit space $\tilde{C}_N(U,V;X,Y)$. Then
\[
\sum_N \frac{\mathcal{G}^N_{S_i,X} \mathcal{G}^N_{S_j, Y}}{a_S(N)} = \sum_N \frac{\vert C_N(i,j;X,Y) \vert}{a_S(N)}
\]
where $C_N(i,j;X,Y)=C_N(S_i,S_j;X,Y)$. Similarly, for a self-dual representation $Z$, let $D_Z(U,V;X,Y)$ be the set of all corners of self-dual exact sequences, as in Figure \ref{fig:crossCorner}. Let $\tilde{D}_Z(U,V;X,Y)$ be the orbit space of the free $Aut_S(Z)$-action on $D_Z(U,V;X,Y)$. The right-hand side of equation \eqref{eq:sdHallIden} becomes
\[
\sum_Z \frac{\mathcal{G}^Y_{S_i,Z} \mathcal{G}^X_{S_j,Z}}{a_S(Z)} =\sum_Z \frac{\vert D_Z(i,j;X,Y) \vert}{a_S(Z)} = \sum_Z \vert \tilde{D}_Z(i,j;X,Y) \vert.
\]

In the notation of Figure \ref{fig:crossCorner}, if $i_U$ and $i_V$ present $U \oplus V$ as an isotropic subrepresentation of $N$, by reducing in stages along $U$ and $V$ we obtain a corner  on $Z = N /\!/ U \oplus V$. In this case, we say that the cross descends to this corner.

\begin{Lem}
\label{lem:noDescend}
If $\mathcal{C} \in C_N(i,j;X,Y)$ does not descend to a corner, then $X \simeq_S Y$.
\end{Lem}
\begin{proof}
The cross fails to descend if and only $\im i_{S_i} + \im i_{S_j}$ is not a two dimensional isotropic subrepresentation. This occurs if $\im  i_{S_i} = \im i_{S_j} $, in which case clearly $X \simeq_S Y$, or if $\im i_{S_i} + \im i_{S_j}$  is non-degenerate, in which case it is isometric to $H(S_i)$. In the latter case $N \simeq_S H(S_i) \oplus X \simeq_S H(S_i) \oplus Y$ and again $X \simeq_S Y$.
\end{proof}

To prove equation \eqref{eq:sdHallIden} we will show that the sum on the right-hand side counts (with weights) crosses that descend to corners while the other two terms count crosses that fail to descend for the two reasons indicated in the proof of Lemma \ref{lem:noDescend}. Since the left-hand side of \eqref{eq:sdHallIden} counts all crosses, the equation will follow.

\begin{Lem}
\label{lem:notEqualNoDes}
There are exactly $a(U) a_S(X)$ crosses in $\tilde{C}_N(U,V;X,Y)$ such that $\im i_U$ and $\im i_V$ intersect trivially and $ \im i_U \oplus \im i_V$ is non-degenerate.
\end{Lem}
\begin{proof}
The assumptions imply $\im i_U \oplus \im i_V \simeq_S H(U)$ and  $N \simeq_S H(U) \oplus X$. Acting by $Aut_S(H(U))$ and $Aut_S(X)$ (both are subgroups of $Aut_S(N)$) we may take $i_U$ to be the standard inclusion $U \rightarrowtail H(U)$, $i_V$ to factor through the standard inclusion $S(U) \rightarrowtail H(U)$ and $\pi_U$ to be the projection $X \oplus \im i_V \twoheadrightarrow X$. The set of pairs $(i_V, \pi_V)$ completing the cross is a $Aut(U)  \times Aut_S(X)$-torsor, with different pairs giving different classes in $\tilde{C}_N(U,V;X,Y)$.
\end{proof}

\begin{Lem}
\label{lem:crossStabilizer}
Let $\mathcal{C} \in C_N(U,V;X,Y)$.
\begin{enumerate}[leftmargin=0cm,itemindent=.6cm,labelwidth=\itemindent,labelsep=0cm,align=left]
\item If $\mathcal{C}$ descends to a corner, then $\text{\textnormal{Stab}}_{Aut_S(N)} (\mathcal{C}) \simeq \Hom(S(U), V)$.
\item If $U =S_i$, $V=S_j$ and $\im i_U + \im i_V$ is non-degenerate, then $\text{\textnormal{Stab}}_{Aut_S(N)} (\mathcal{C})$ is trivial.
\end{enumerate}
\end{Lem}

\begin{proof}
Suppose that $\mathcal{C}$ descends to a corner and let $\phi \in \textnormal{Stab}_{Aut_S(N)} (\mathcal{C})$. From the proof of Lemma \ref{lem:stab}, the restrictions $\phi_{\vert E_U}$ and $\phi_{\vert E_V}$ factor through maps $X \rightarrow U$ and $Y \rightarrow V$, respectively. As $\phi$ also stabilizes the induced self-dual exact sequence
\[
0 \rightarrow U \oplus V \xrightarrow[]{i_U \oplus i_V} N \dashrightarrow Z \rightarrow 0,
\]
the restriction of $\phi$ to $E_{U \oplus V} = E_U \cap E_V$ factors through a map $Z \rightarrow U \oplus V$. Compatibility with $\phi_{\vert E_U}$ and $\phi_{\vert E_V}$ requires that this map vanish. Then $\phi$ is uniquely determined by an element of $\Hom(S(U \oplus V), U \oplus V)^{-S}$. Again, compatibility with $\phi_{\vert E_U}$ and $\phi_{\vert E_V}$ imply that only the summand
\[
(\Hom(S(U),V) \oplus \Hom(S(U),V) )^{-S} \simeq \Hom(S(U),V)
\]
contributes to $\phi$. Reversing this argument shows that each element of $\Hom(S(U),V)$ gives rise to an element of $\mbox{Stab}_{Aut_S(N)} (\mathcal{C})$. The proof of the second statement is similar. 
\end{proof}

\begin{Prop}
\label{prop:fibreCard}
There are exactly $\vert \Ext^1(S(U), V) \vert$ elements of $\bigsqcup_N \tilde{C}_N(U,V; X,Y)$ that descend to each element of $\bigsqcup_Z \tilde{D}_Z(U,V; X,Y)$. 
\end{Prop}
\begin{proof}
Given a corner as in Figure \ref{fig:crossCorner}, consider the pullback of $\tilde{\pi}_U$ along $\tilde{\pi}_V$
\[
\begin{tikzpicture}[every node/.style={on grid},  baseline=(current bounding box.center),node distance=1.4]
  \node (A) {$V$}; 
  \node (B) [right=of A] {$V$};
  \node (C) [below=of A] {$\mathbb{E}$};
  \node (D) [right=of C] {$\tilde{E}_V$};
  \node (E) [below=of C] {$\tilde{E}_U$};
  \node (F) [right=of E] {$Z$};
  \node (G) [left=of C] {$U$};
  \node (H) [left=of E] {$U$};  
  
  \draw[double equal sign distance] (A)--(B);
  \draw[double equal sign distance] (G)--(H);
  \draw[>->] (A)-- node [left] {\footnotesize $j^{\prime}_V$ } (C);
  \draw[>->] (B)-- node [right] {\footnotesize $\tilde{j}_V$ } (D);
  \draw[->>] (D)-- node [right] {\footnotesize $\tilde{\pi}_V$ } (F);
  \draw[->>] (C)-- node [left] {\footnotesize $\pi^{\prime}_V$ } (E);
  \draw[->>] (E)-- node [below] {\footnotesize $\tilde{\pi}_U$ } (F);
  \draw[->>] (C)-- node [above] {\footnotesize $\pi^{\prime}_U$ } (D);
  \draw[>->] (G)-- node [above] {\footnotesize $j^{\prime}_U$ } (C);
  \draw[>->] (H)-- node [below] {\footnotesize $\tilde{j}_U$ } (E);
\end{tikzpicture}
\]
and the resulting short exact sequence
\[
U \oplus V \rightarrowtail  \mathbb{E} \twoheadrightarrow Z.
\]
From Lemma \ref{lem:fibre}, the group $\Ext^1(S(U \oplus V), U \oplus V)^S$ acts transitively on the set of lifts of this sequence to self-dual extensions of $Z$ by $U \oplus V$. Pick such a lift and act by the subgroups $\Ext^1(S(U),U)^S$ and $\Ext^1(S(V),V)^S$ to modify it so that the reductions of the central term $N$ by $U$ and $V$ are isometric to $X$ and $Y$, respectively. By construction, $N$, viewed as a cross, descends to the original corner. The orthogonals of $U$ and $V$ in $N$ give rise to two commutative diagrams
\begin{equation}
\label{diag:commPair}
\begin{tikzpicture}[every node/.style={on grid},  baseline=(current bounding box.center),node distance=1.4]
  \node (A) {$U$};
  \node (B) [right=of A]{$\mathbb{E}$}; 
  \node (C) [right=of B]{$\tilde{E}_V$};
  \node (D) [below=of A]{$U$};
  \node (E) [right=of D]{$E_U$}; 
  \node (F) [right=of E]{$X$};
  \node (G) [below=of E]{$S(V)$};
  \node (H) [right=of G]{$S(V)$};
  
  \draw[>->] (A)-- node [above] {\footnotesize $j^{\prime}_U$ }(B);
  \draw[->>] (B)--  node [above] {\footnotesize $\pi^{\prime}_U$ } (C);
  \draw[>->] (D)-- (E);
  \draw[->>] (E)-- node [above] {\footnotesize $\pi_U$ }  (F);
  \draw[double equal sign distance] (A)-- (D);
  \draw[double equal sign distance] (G)-- (H);
  \draw[>->] (B)--  node [left] {\footnotesize $l_U$ }  (E);
  \draw[>->] (C)--   (F);
  \draw[->>] (E)--   (G);
  \draw[->>] (F)--  (H);
  \end{tikzpicture}
\hspace{55pt}
\begin{tikzpicture}[every node/.style={on grid},  baseline=(current bounding box.center),node distance=1.4]
  \node (A) {$V$};
  \node (B) [right=of A]{$\mathbb{E}$}; 
  \node (C) [right=of B]{$\tilde{E}_U$};
  \node (D) [below=of A]{$V$};
  \node (E) [right=of D]{$E_V$}; 
  \node (F) [right=of E]{$Y$};
  \node (G) [below=of E]{$S(U)$};
  \node (H) [right=of G]{$S(U)$};
  
  \draw[>->] (A)-- node [above] {\footnotesize $j^{\prime}_V$ }(B);
  \draw[->>] (B)--  node [above] {\footnotesize $\pi^{\prime}_V$ } (C);
  \draw[>->] (D)--  (E);
  \draw[->>] (E)-- node [above] {\footnotesize $\pi_V$ }  (F);
  \draw[double equal sign distance] (A)-- (D);
  \draw[double equal sign distance] (G)-- (H);
  \draw[>->] (B)-- node [left] {\footnotesize $l_V$ }   (E);
  \draw[>->] (C)--   (F);
  \draw[->>] (E)--   (G);
  \draw[->>] (F)--  (H);
  \end{tikzpicture}
\end{equation}
Observe that the only data in the left (say) diagram not determined by the corner is $(E_U; l_U, \pi_U)$ and that this data is determined only up to automorphisms of $E_U$. Moreover, the pushout of $l_U$ along $l_V$ recovers $N$ up to isomorphism. Therefore, to count classes of crosses that lift the corner it suffices to count classes of pairs $(E_U; l_U, \pi_U)$ and $(E_V; l_V, \pi_V)$ that make diagrams \eqref{diag:commPair} commute and that are compatible in the sense that the resulting central term $N$ admits a self-dual structure.   From \cite[\S 9.3.8.b]{grothendieck1972} (see also \cite{green1995}) the classes of tuples $(E_U; l_U, \pi_U)$ making the left diagram commute is an $\Ext^1(S(V),U)$-torsor. Similarly, the data for the right diagram is an $\Ext^1(S(U),V)$-torsor. Compatibility requires these group actions be dependent. Precisely, only the subgroup
\[
\Ext^1(S(U),V) \simeq (\Ext^1(S(V),U) \oplus \Ext^1(S(U),V)^S   \subset   \Ext^1(S(U \oplus V), U \oplus V)^S
\]
preserves the condition that the central term is self-dual. This shows that the set of crosses in $\bigsqcup_N \tilde{C}_N(U,V; X,Y)$ lifting the original corner is a torsor for $\Ext^1(S(U),V)$, and completes the proof.
\end{proof}

\begin{proof}[Completion of the proof of Theorem \ref{thm:redModuleStructure}]
Write
\[
C_N(i,j;X,Y) = C^{(1)}_N(i,j;X,Y) \bigsqcup C^{(2)}_N(i,j;X,Y)
\]
with $C^{(1)}_N(i,j;X,Y)$ the set of crosses that descend to corners. Burnside's lemma and the first part of Lemma \ref{lem:crossStabilizer} give
\[
\sum_N \frac{ \vert C_N(i,j;X,Y) \vert}{a_S(N)} = \sum_N  \frac{\vert \tilde{C}^{(1)}_N(i,j;X,Y)\vert}{\vert \Hom (S_{\sigma(i)}, S_j) \vert}  +  \sum_N \frac{\vert C^{(2)}_N(i,j;X,Y) \vert}{a_S(N)}.
\]
By Proposition \ref{prop:fibreCard} the first sum is
\[
\sum_N  \frac{\vert \tilde{C}^{(1)}_N(i,j;X,Y)\vert}{\vert \Hom (S_{\sigma(i)}, S_j) \vert} = \frac{\vert \Ext^1  (S_{\sigma(j)}, S_i) \vert }{\vert \Hom(S_{\sigma(j)}, S_i) \vert} \sum_Z  \vert \tilde{D}_Z(i,j;X,Y)\vert
\]
while Lemma \ref{lem:notEqualNoDes} and the second part of Lemma \ref{lem:crossStabilizer} give for the second sum
\[
\sum_N \frac{\vert C^{(2)}_N(i,j;X,Y) \vert}{a_S(N)} = \delta_{X,Y}\delta_{i,j}a(S_i) a_S(X) \sum_N \frac{\mathcal{G}^N_{S_i,X}}{a_S(N)}  + \delta_{X,Y}\delta_{i, \sigma(j)} a(S_i) a_S(X).
\]
Here the crosses counted by the first (resp. second) term on the right-hand side fail to descend because $\im i_{S_i} = \im i_{S_j}$ (resp. $\im i_{S_i} + \im i_{S_j}$ is non-degenerate). Finally, using Theorem \ref{thm:sdRiedtmann} to evaluate $\sum \frac{\mathcal{G}^N_{S_i,X}}{a_S(N)}$ establishes equation \eqref{eq:sdHallIden}.
\end{proof}

We now turn to the decomposition of $\mathcal{M}_Q$ into irreducible $B_{\sigma}(\mathfrak{g}_Q)_{\nu_0}$-modules.

\begin{Def}
A non-zero element $\xi \in \mathcal{M}_Q$ is called cuspidal if $E_i \xi =0$ for all $i \in Q_0$.
\end{Def}

Fix a homogeneous orthogonal basis $\mathcal{C}_Q$ for the $R$-module of cuspidals. Given $\xi \in \mathcal{C}_Q$, define a $\sigma$-invariant weight $\lambda_{\xi}$ by
\[
\lambda_{\xi}(\epsilon_i)=- (\mathbf{dim}\, \xi,\epsilon_i) - \mathcal{E}(\epsilon_i) - \mathcal{E}(\epsilon_{\sigma(i)}).
\]

\begin{Thm}
\label{thm:hallModDecomp}
The Hall module $\mathcal{M}_Q$ admits an orthogonal direct sum decomposition into irreducible highest weight $B_{\sigma}(\mathfrak{g}_Q)_{\nu_0}$-modules generated by $\mathcal{C}_Q$:
\[
\mathcal{M}_Q = \bigoplus_{\xi \in \mathcal{C}_Q} V_{\sigma}(\lambda_{\xi})_{\nu_0}.
\]
\end{Thm}
\begin{proof}
We first show that the submodule $\langle \xi \rangle \subset \mathcal{M}_Q$ generated by $\xi \in \mathcal{C}_Q$ is isomorphic to $V_{\sigma}(\lambda_{\xi})_{\nu_0}$. Indeed, suppose that $x \in \langle \xi \rangle$ is non-zero with $E_i x=0$ for all $i \in Q_0$. If $x= \sum_{i\in Q_0} F_i y_i$ for some $y_i \in \langle \xi \rangle$, then
\[
( x, x )_{\mathcal{M}} = \sum_{i \in Q_0} ( x, F_i y_i )_{\mathcal{M}} = \frac{1}{\nu^{-2}-1} \sum_{i \in Q_0} ( E_i x, y_i )_{\mathcal{M}} =0.
\]
However, writing $x$ in the natural basis of $\mathcal{M}_Q$ as $x = \sum_M c_M [M]$ shows
\[
(x, x)_{\mathcal{M}} = \sum_M \frac{c_M^2}{a_S(M)} >0,
\]
a contradiction. So, $x$ is a scalar multiple of $\xi$ and Proposition \ref{prop:Vsigma} implies $\langle \xi \rangle \simeq V_{\sigma}(\lambda_{\xi})_{\nu_0}$.

Suppose now that $\xi_1, \xi _2 \in \mathcal{C}_Q$ are distinct and therefore orthogonal. It follows that $\langle \xi_1 \rangle$ and $\langle \xi_2 \rangle$ are also orthogonal. Then we have an inclusion
\[
\bigoplus_{\xi \in \mathcal{C}_Q}V_{\sigma}(\lambda_{\xi})_{\nu_0} \hookrightarrow \mathcal{M}_Q.
\]
To prove that this is an isomorphism, note that the restriction of $( \cdot, \cdot)_{\mathcal{M}}$ to the submodule generated by $\mathcal{C}_Q$ is non-degenerate. Let $0 \neq x \in \mathcal{M}_Q$ be orthogonal to this submodule and of minimal dimension with this property. As $x$ is not cuspidal, $E_i x \neq 0$ for some $i \in Q_0$. By the minimality assumption, $E_i x \in \bigoplus_{\xi \in \mathcal{C}_Q} V_{\sigma}(\lambda_{\xi})_{\nu_0}$. Since $F_i E_i x \in \bigoplus_{\xi \in \mathcal{C}_Q} V_{\sigma}(\lambda_{\xi})_{\nu}$,
\[
(E_i x, E_i x)_{\mathcal{M}} = (\nu^{-2}-1)(x, F_i E_i x)_{\mathcal{M}} =0,
\]
contradicting $E_i x \neq 0$, completing the proof.
\end{proof}

As a special case of Theorem \ref{thm:hallModDecomp}, note that in the orthogonal case we have $\langle [0] \rangle \simeq V_{\sigma}(0)_{\nu_0}$. A geometric version of this isomorphism was obtained by Enomoto \cite{enomoto2009} by studying perverse sheaves on the moduli stack of orthogonal representations over $\mathbb{C}$. Moreover, a lower global basis of $V_{\sigma}(0)$ was obtained, giving an orthogonal analogue of Lusztig's construction of the lower global basis of $U_v^-(\mathfrak{g}_Q)$ \cite{lusztig1990}. In \cite{varagnolo2011} Enomoto's approach was generalized to construct lower global bases of $V_{\sigma}(\lambda)$ for general $\lambda$.

A result stronger than Theorem \ref{thm:redModuleStructure}, but valid only for symplectic, orthogonal and unitary representations of the Jordan quiver, was proved in \cite{leeuwen1991}. Note that in these cases $\mathcal{E}$ is identically zero. Following Zelevinsky \cite{zelevinsky1981} and interpreting $\mathcal{M}_Q$ in terms of unipotent characters of classical groups, van Leeuwen constructed a ring homomorphism $\Phi : \mathcal{H}_Q \rightarrow \mathcal{H}_Q \otimes_R \mathcal{H}_Q$, third order in the Hall numbers, satisfying $\rho ([U] \star [M]) = \Phi([U]) \star \rho([M])$. See also \cite{tadic1995} for a $p$-adic version of this result.  Theorem \ref{thm:redModuleStructure} recovers a particular component of this $\Phi$-twisted Hopf module structure.\footnote{While Theorem \ref{thm:redModuleStructure} is stated for loopless quivers, the verification of \eqref{eq:sdHallIden} above holds without this assumption.} It would be very interesting to extend this result to arbitrary $(Q,\sigma)$.

\section{Hall modules of finite type quivers}

\label{sec:ftHallMod}

\subsection{Classification of self-dual representations over finite fields}
A quiver $Q$ is called finite type if it has only finitely many isomorphism classes of indecomposable representations over any field. By Gabriel's theorem \cite{gabriel1972}, a connected finite type quiver is an orientation of an $ADE$ Dynkin diagram and its indecomposables are in bijection with a set of positive roots of $\mathfrak{g}_Q$.

\begin{Ex}
Let $Q$ be an orientation of $A_{2n}$ or $A_{2n+1}$. Label the nodes $-n,\dots, n$ (omitting $0$ for $A_{2n}$) with $i$ and $i+1$ adjacent. Let $I_{i,j}$ be the indecomposable with dimension vector $\epsilon_i  + \cdots + \epsilon_j$ and all intermediate structure maps the identity. Then $\{ I_{i,j} \}_{-n \leq i \leq j \leq n}$ is a complete set of representatives of indecomposables. 
\end{Ex}

Similarly, $(Q, \sigma)$ is called finite type if it has only finitely many isometry classes of indecomposable self-dual representations over any field whose characteristic is not two. By \cite[Theorem 3.1]{derksen2002}, $(Q, \sigma)$ is finite type if and only if $Q$ is finite type. In \textit{loc. cit.} the authors work with orthogonal and symplectic representations but their proof applies to the more general dualities considered here. It follows that if $(Q,\sigma)$ is finite type and not a disjoint union of quivers with involution, then $Q$ is of Dynkin type $A$ or $Q=Q^{\prime \sqcup}$ with $Q^{\prime}$ of Dynkin type $ADE$.

For the purpose of studying Hall modules of finite type quivers it suffices to restrict attention to orthogonal, symplectic and unitary representations. Indeed, any choice of duality is equivalent to one of these three choices.

\begin{Lem}
\label{lem:splitting}
\begin{enumerate}[leftmargin=0cm,itemindent=.6cm,labelwidth=\itemindent,labelsep=0cm,align=left]
\item The representation underlying a self-dual indecomposable is either indecomposable or of the form $I \oplus S(I)$ for some indecomposable $I$.
\item Let $Q$ be finite type. If an indecomposable $I$ does not admit a self-dual structure, then, up to isometry, $H(I)$ is the unique self-dual structure on $I \oplus S(I)$.
\end{enumerate}
\end{Lem}
\begin{proof}
The first statement is given in \cite[Proposition 2.7]{derksen2002} for algebraically closed fields but the proof works without this assumption.

Since $Q$ is finite type, there is a total order $\preceq$ on the set of indecomposables such that $\hbox{Hom}(I, J) = \hbox{Ext}^1(J,I)=0$ if $J \prec I$; see \cite{crawley1992}. Writing a self-dual structure $\psi$ on $I \oplus S(I)$ as
\[
I \oplus S(I) \xrightarrow[]{\left( \begin{smallmatrix}  a & b \\ c & d \end{smallmatrix} \right) } S(I) \oplus S^2(I)
\]
we see that  $S(a) \Theta_I=a$.  If $I \simeq S(I)$, then $\Hom(I,S(I)) \simeq k$ and $a=0$; otherwise $a$ is a self-dual structure on $I$. Similarly $d=0$, and it is now straightforward to verify that $\psi$ is isometric to $H(I)$. If instead $I \not\simeq S(I)$, we may assume $S(I) \prec I$. Again $a=0$ and acting by $Aut(I)$ we may take $b=1_{S(I)}$ and $c=\Theta_I$. Then $\left( \begin{smallmatrix} 1 & -\frac{1}{2}d \\ 0 & 1 \end{smallmatrix} \right)$ is an isometry from $\psi$ to $H(I)$.
\end{proof}

We can use Lemma \ref{lem:splitting} to describe the self-dual indecomposables of finite type quivers over finite fields. For $Q^{\sqcup}$ the self-dual indecomposables are in bijection with the indecomposables of $Q$. Indecomposables in type $A_{2n+1}$ (resp. $A_{2n}$) do not admit symplectic (resp. orthogonal) structures. In these cases the hyperbolics $\{ H(I_{i,j}) \}$ are therefore a complete set of representatives of self-dual indecomposables. For orthogonal (resp. symplectic) representations in type $A_{2n+1}$ (resp. $A_{2n}$), the indecomposables $I_{-i,i}$ admit two self-dual structures, denoted by $R_i^c$ according to the following rule. By composing the (inverse) structure maps of $R_i^c$ we get a map from the $(-i)$th vector space to the $i$th vector space. Using the self-dual structure, this gives an orthogonal form on the $(-i)$th vector space, whose Witt class we denote by $c \in W(\mathbb{F}_q)$. In these cases we therefore replace $H(I_{-i,i})$ in the above set with the two $R_i^c$. Finally, $I_{-i,i}$ admits a unique unitary structure $R_i$, which replaces $H(I_{-i,i})$.

We introduce some notation for the Witt ring $W(\mathbb{F}_q)$. The subset $W_1$ spanned by classes of one dimensional forms is stable under tensor product and, as a multiplicative group, we identify it with $\{1,-1\}$. Also, denote by $\eta: \mathbb{F}_q^{\times} \rightarrow \{1,-1\} \simeq W_1$ the quadratic character.

\begin{Ex}
There are six indecomposable orthogonal representations of $\begin{tikzpicture}[thick,scale=.33,decoration={markings,mark=at position 0.6 with {\arrow{>}}}]
\draw[postaction={decorate}] (0,0) to  (2,0);
\draw[postaction={decorate}] (2,0) to  (4,0);
\fill (0,0) circle (4pt);
\fill (2,0) circle (4pt);
\fill (4,0) circle (4pt);
\end{tikzpicture}$:
\[
H(S_1): \mathbb{F}_q \rightarrow 0 \rightarrow \mathbb{F}_q, \;\;\;\;\;\;\;\; H(I_{0,1}): \mathbb{F}_q \xrightarrow[]{\left( \begin{smallmatrix} 1 \\ 0 \end{smallmatrix} \right)} \mathbb{F}_q^2 \xrightarrow[]{\left( \begin{smallmatrix} 0 & -1 \end{smallmatrix} \right) } \mathbb{F}_q,
\]
\[
R_0^{c}: 0 \rightarrow \mathbb{F}_q \rightarrow 0, \;\;\;\;\;\;\; \;\;\;\;\;\;\;\;\; R_1^{\eta(-1)c}: \mathbb{F}_q \xrightarrow[]{1} \mathbb{F}_q \xrightarrow[]{-\tilde{c}} \mathbb{F}_q .
\]
The orthogonal form on the central node is hyperbolic for $H(I_{0,1})$ and has Witt index $c$ for $R_0^c$ and $R_1^{\eta(-1)c}$. The element $\tilde{c} \in \mathbb{F}_q$ satisfies $\eta(\tilde{c}) =c$.
\end{Ex}

Over algebraically closed fields self-dual indecomposables of finite type quivers admit a partial interpretation in terms of root systems \cite{derksen2002}. The above classification extends this interpretation to finite fields. Denote by $\mathcal{I}_Q^{\mathfrak{g}}$ the multiset of dimension vectors of self-dual $\mathbb{F}_q$-indecomposables, with $\mathfrak{g}$ one of $\mathfrak{so}$, $\mathfrak{sp}$ or $\mathfrak{u}$.

\begin{Thm}
\label{thm:sdGabriel}
Let $(Q, \sigma)$ be finite type. Then $\mathcal{I}_Q^{\mathfrak{g}}$ is independent of the orientation of $Q$ and the finite field $\mathbb{F}_q$. Explicitly,
\begin{enumerate}
[leftmargin=0cm,itemindent=.6cm,labelwidth=\itemindent,labelsep=0cm,align=left]
\item if $Q$ is of type $A_{2n}$, then $\mathcal{I}_Q^{\mathfrak{so}}$ is in bijection with $BC_n^+$ while $\mathcal{I}_Q^{\mathfrak{sp}}$ (resp. $\mathcal{I}_Q^{\mathfrak{u}}$) surjects onto $B_n^+$, the short roots having fibre of cardinality three (resp. two),

\item if $Q$ is of type $A_{2n+1}$, then $\mathcal{I}_Q^{\mathfrak{sp}}$ (resp. $\mathcal{I}_Q^{\mathfrak{u}}$) is in bijection with $C_{n+1}^+$ (resp. $B_{n+1}^+$), while $\mathcal{I}_Q^{\mathfrak{so}}$ surjects onto $B_{n+1}^+$, the short roots having fibre of cardinality two, and

\item $\mathcal{I}_{Q^{\sqcup}}^{\mathfrak{g}}$ is in bijection with the set of positive roots of $\mathfrak{g}_Q$.
\end{enumerate}
\end{Thm}

\begin{proof}
Suppose that $Q$ is of type $A_{2n}$. Our notation for roots systems is
\[
B_n^+ = \{ \varepsilon_i  \pm \varepsilon_j , \varepsilon_i \; \vert \;  0 \leq i \leq j \leq n-1 \}
\]
and $BC_n^+ = B_n^+ \bigsqcup \{ 2 \varepsilon_i \}_{i =0}^{n-1}$.  For orthogonal representations, the bijection is
\[
H(I_{i,j}) \mapsto \left\{ \begin{array}{ll}  \varepsilon_{n-j} - \varepsilon_{n-i+1},& \hbox{for } 1 \leq i \leq j \leq n,  \\ 
 \varepsilon_{n-j} + \varepsilon_{n+i} ,  & \hbox{for } - n \leq i < -1  \hbox{ and }  1 \leq j \leq n. \end{array} \right.
\]
For symplectic (resp. unitary) representations the bijection is as above, but now $R_i^c$ (resp. $R_i$) maps to $\varepsilon_{n-i}$. The other cases are similar.
\end{proof}

\subsection{Application to Hall modules}

A weak version of the Krull-Schmidt theorem holds for self-dual representations: a self-dual representation decomposes into an orthogonal direct sum of self-dual indecomposables. However, while the type $c \in W(\mathbb{F}_q)$ of a summand\footnote{The indexing convention generalizes that for $r=1$: the underlying representation of $R_i^{\oplus r, c}$ is $I_{-i,i}^{\oplus r}$ and induced orthogonal form on the $(-i)$th vector space has type $c$.} $R_i^{\oplus r, c}$ is well-defined, the type of its indecomposable summands may not be.

\begin{Prop}
\label{prop:reducDecomp}
Let $(Q, \sigma)$ and $(Q^{\prime}, \sigma)$ be finite type quivers with involution with the same underlying graph and duality. Then the decompositions of $\mathcal{M}_Q$ and $\mathcal{M}_{Q^{\prime}}$ into irreducible $B_{\sigma}(\mathfrak{g}_Q)_{\nu_0}$-modules coincide.
\end{Prop}

\begin{proof}
Let $\mbox{ch}(\mathcal{M}_Q)$ be the generating function of the ranks of the $T_i$-weight spaces of $\mathcal{M}_Q$. Note that the $T_i$-weight of a self-dual representation depends only on its dimension vector and not on the orientation of $Q$. Theorem \ref{thm:sdGabriel} and the weak Krull-Schmidt theorem therefore imply $\mbox{ch}(\mathcal{M}_Q) =\mbox{ch}(\mathcal{M}_{Q^{\prime}})$. Since $Q$ is finite type, the symmetrized Euler form is non-degenerate. It follows that the weight $\lambda$ subspace of $V_{\sigma}(\lambda)$ is rank one. From this we conclude that the characters $\{\mbox{ch}(V_{\sigma}(\lambda))\}_{\lambda \in \mbox{\footnotesize Hom}(\Phi, \frac{1}{2} \mathbb{Z})}$ are linearly independent. The proposition now follows.
\end{proof}

We first deal with those $(Q, \sigma)$ admitting only hyperbolics.

\begin{Thm}
\label{thm:noKFormIrred}
If a duality structure on a finite type quiver $(Q, \sigma)$ admits only hyperbolic self-dual representations, then $\mathcal{M}_Q = \langle [0] \rangle  \simeq V_{\sigma}(\lambda_{[0]})_{\nu_0}$.
\end{Thm}

\begin{proof}
By assumption, an arbitrary self-dual representation is of the form
\[
H(U) \simeq_S \bigoplus_{i=1}^l H(I_i)^{\oplus m_i}, \;\;\; m_i \geq 0
\] 
for indecomposables $I_i$ satisfying $I_i \not \simeq I_j$ and $I_i \not\simeq S(I_j)$ for $i \neq j$. Without loss of generality we may assume $S(I_i) \preceq I_i \prec I_{i+1} \prec \cdots \prec I_l$ for $i=1, \dots, l$. This implies $\Ext^1(S(I_i),I_j)=0$ for all $i \leq j$, and by duality, also for $ i \geq j$. Hence $\Ext^1(S(U),U)=0$ and
\[
[U] \star [0] = \nu^{-\mathcal{E}(U)} G^{H(U)}_{U,0} [H(U)].
\]
The equality $\mathcal{M}_Q = \langle [0] \rangle$ now follows from the fact $[S_i]$, $i\in Q_0$, generate $\mathcal{H}_Q$.
\end{proof}

Hall modules of unitary, symplectic and orthogonal representations in type $A_n$, $A_{2n}$ and $A_{2n+1}$, respectively, are not covered by Theorem \ref{thm:noKFormIrred}. To deal with these cases we will use the following simple fact, dual to the statement that the Hall algebra of a finite type quiver is generated by simple representations.

\begin{Lem}
\label{lem:cuspPrim}
Let $(Q,\sigma)$ be a finite type quiver. Then $\xi \in \mathcal{M}_Q$ is cuspidal if and only if $\rho(\xi) = [0] \otimes \xi$.
\end{Lem}

Denote by $\vec{A}_{2n}$ and $\vec{A}_{2n+1}$ the Dynkin diagrams with orientation $-n \rightarrow  \cdots \rightarrow n$. Combined with Proposition \ref{prop:reducDecomp} and Theorem \ref{thm:noKFormIrred}, the next result completes the decomposition of finite type Hall modules into irreducible representations.

Given $\underline{c} = (c_j)_{j \in J} \in W_1^J$, we write $R^{\underline{c}} = \oplus_{j \in J} R_j^{c_j}$.

\begin{Thm}
\label{thm:noHypCusp}
Homogeneous bases for the submodules of cuspidals are as follows:
\begin{enumerate}
\item $\mathcal{C}^{\mathfrak{u}}_{\vec{A}_{2n}} = \{ [ 0 ] \}$ and $\mathcal{C}^{\mathfrak{u}}_{\vec{A}_{2n+1}} = \{ [0], [R_0 ] \}$;
\item $\mathcal{C}^{\mathfrak{sp}}_{\vec{A}_{2n}} = \{[0], \xi_1, \dots, \xi_n\}$, where $\xi_j = \sum_{\underline{c} \in W_1^{[1,j]}} a_{\underline{c}} [R^{\underline{c}}]$ and $a_{\underline{c}} = \prod_{i \, \hbox{\footnotesize odd}} c_i$; 

\item $\mathcal{C}^{\mathfrak{so}}_{\vec{A}_{2n+1}} = \{ [0], \xi_0^b,  \dots, \xi_n^b \}$,
where $\xi^b_j = \sum_{\underline{c}} a_{\underline{c}} [R^{\underline{c}}]$, $a_{\underline{c}}$ is as above and the sum is over all $\underline{c}  \in W_1^{[0,j]}$ satisfying $\sum_{i = 0 }^j c_i =b \in W(\mathbb{F}_q)$.
\end{enumerate}
\end{Thm}

\begin{proof}
Fix the following choice of total order $\prec$:
\[
I_{i,j} \prec I_{k,l} \mbox{ if and only if } i >k \mbox{ or } i=k \mbox{ and } j \geq l.
\]
Then $S(I_{i,j}) \preceq I_{i,j}$ if and only if $i + j \leq 0$.

Consider $N = H(U) \oplus R$, where $R$ has no hyperbolic summands and $U$ is written as in the proof of Theorem \ref{thm:noKFormIrred}. Then we have $\Ext^1 (I_{-k,k}, I_{i,j}) =0$ for all $k$ and $i + j \leq 0$, as can be verified directly. Since the underlying representation of $R$ is a direct sum of indecomposables of the form $I_{-k,k}$, $\Ext^1(R,U) = 0$ and dually $\Ext^1(S(U), R)=0$. It follows that $N$ is the only self-dual extension of $R$ by $U$. By Lemma \ref{lem:cuspPrim}, if $[N]$ appears with non-zero coefficient in a cuspidal then $U=0$.

In the unitary case a homogeneous cuspidal must then be of the form $\xi = [ \oplus_{j \in J} R_j]$. But $E_j \xi \neq 0$ whenever $0 \neq j \in J$. Hence, either $J=\emptyset$ or $Q=\vec{A}_{2n+1}$ and $J =\{0\}$, and we arrive at the claimed bases of cuspidals.

Consider now a homogeneous cuspidal $\xi \in \mathcal{M}_{\vec{A}_{2n}}^{\mathfrak{sp}}$. Then $\xi$ contains no terms of the form $[R \oplus R_i^{\oplus 2, c}]$, where $c \neq 0$ (so that $R_i^{\oplus 2, c}$ is not hyperbolic) and $R$ contains no $R_i$ summand. Otherwise, $[S_i^{\oplus 2}] \otimes [R \oplus R_{i-1}^{\oplus 2, c}]$ would appear in $\rho(\xi)$. Therefore,
\[
\xi = \sum_{\underline{c} \in W_1^J} a_{\underline{c}} [R^{\underline{c}}]
\]
for some $J \subset [1, n ]$ and $a_{\underline{c}} \in \mathbb{Q}$. The orientation of $\vec{A}_{2n}$ ensures that $E_i \xi=0$ for $i \leq 0$, while a short calculation shows
\[
R_i^{c} /\!/ S_i \simeq_S R_{i-1}^{\eta(-1)c}, \;\;\; i > 0.
\]
This implies that if $2 \leq i \in J$, then $i -1 \in J$, as otherwise $E_i \xi \neq 0$. So, if non-empty, $J = [1, j]$ for some $1 \leq j \leq n$. The condition $E_1 \xi=0$ is equivalent to $a_{\underline{c}}=-a_{\underline{c}^{\prime}}$ whenever $\underline{c}$ and $\underline{c}^{\prime}$ agree except in their first slot. For $2 \leq i \leq j$, the condition $E_i \xi=0$ is equivalent to $a_{\underline{c}}=-a_{\underline{c}^{\prime}}$ whenever $\underline{c}$ and $\underline{c}^{\prime}$ agree except in their $(i-1)$th and $i$th slots and satisfy (in $W(\mathbb{F}_q)$)
\[
c_{i-1} + \eta(-1) c_i = c^{\prime}_{i-1} +  \eta(-1) c^{\prime}_i.
\]
It is straightforward to verify that, up to a non-zero scalar multiple, $a_{\underline{c}}$ must be as claimed.

The argument for the final case is similar. The index $b \in W(\mathbb{F}_q)$ labels the Witt summand of $\mathcal{M}_{\vec{A}_{2n+1}}^{\mathfrak{so}}$ in which $\langle \xi_j^b \rangle$ lies; see Proposition \ref{prop:wittGradings}.
\end{proof}

\begin{Ex}
The Hall module of orthogonal representations of $\vec{A}_3$ has five irreducible summands, with cuspidal generators $[0]$, $[R_0^1]$, $[R_0^{-1}]$ and
\[
[R_0^1 \oplus R_1^1]  - [R_0^{-1} \oplus R_1^{-1}]  \;\;\;, \;\;\;\; [R_0^{-1} \oplus R_1^1]  - [R_0^1 \oplus R_1^{-1}] .
\]
\end{Ex}

\footnotesize

\bibliographystyle{plain}
\bibliography{mybib}

\end{document}